\documentclass[12pt,a4paper,reqno]{article}

\usepackage{amsmath,amssymb,amsthm,amscd,dsfont,psfrag,color}

\usepackage[english]{babel}
\usepackage[mathscr]{eucal}
\usepackage[dvips]{graphicx}

\font\sc=rsfs10 at 12pt

\input{comment.sty}
\includecomment{task}
\includecomment{MM}
\includecomment{JK}
\numberwithin{equation}{section}
 
\renewcommand{\a}{\alpha}
\renewcommand{\b}{\beta}
\newcommand{\g}{\gamma}
\newcommand{\G}{\Gamma}
\renewcommand{\d}{\delta}
\newcommand{\D}{\Delta}

\newcommand{\ve}{\varepsilon}
\newcommand{\z}{\zeta}
\newcommand{\y}{\eta}
\renewcommand{\th}{\theta}

\renewcommand{\l}{\lambda}
\renewcommand{\L}{\Lambda}
\newcommand{\m}{\mu}

\newcommand{\x}{\xi}

\renewcommand{\r}{\rho}
\newcommand{\s}{\sigma}

\newcommand{\f}{\phi}

\newcommand{\F}{\Phi}
\newcommand{\h}{\chi}
\newcommand{\p}{\psi}

\renewcommand{\O}{\Omega}

\newcommand{\C}{{\mathbb C}}
\newcommand{\R}{{\mathbb R}}
\newcommand{\Z}{{\mathbb Z}}
\newcommand{\N}{{\mathbb N}}

\newcommand{\mbf}[1]{\protect\mbox{\boldmath$#1$\unboldmath}}
\newcommand{\lmbf}[1]{\protect\mbox{{\scriptsize\boldmath$#1$\unboldmath}}}

\newcommand{\Hb}{{\mathbf H}}

\newcommand{\Lb}{{\mathbf L}}

\newcommand{\tF}{\mathfrak t}
\newcommand{\TF}{\mathfrak T}

\newcommand{\Ac}{{\mathcal A}}
\newcommand{\Bc}{{\mathcal B}}

\newcommand{\Hc}{{\mathcal H}}

\newcommand{\Oc}{{\mathcal O}}

\newcommand{\Rc}{{\mathcal R}}

\newcommand{\Ssf}{{\mathsf S}}
\newcommand{\Tsf}{{\mathsf T}}

\newcommand{\pd}{\partial} 

\newcommand{\pdf}[2]{\frac{\partial {#1}}{\partial {#2}}}

\newcommand{\CH}{\mathcal{H}}

\DeclareMathOperator{\im}{{\rm Im}\,}
\DeclareMathOperator{\re}{{\rm Re}\,}
\DeclareMathOperator{\rank}{rank}

\newcommand{\coun}{\operatorname{Count\,}}

\newcommand{\dom}{\operatorname{Dom\,}}

\newcommand{\supp}{\operatorname{supp\,}}

\newcommand{\Res}{\operatorname{Res\,}}
\newcommand{\dist}{\operatorname{dist\,}}

\newcommand{\opw}[1]{{\rm op}^{W}[{#1}]}

\newcommand{\spec}{\operatorname{spec\,}}

\newcommand{\sme}{\operatorname{spec}_{\operatorname{ess}}}
\newcommand{\smd}{\operatorname{spec}_{\operatorname{d}}} 
 
\newcommand{\bop}[2]{\mathcal{B}({#1},{#2})}
\newcommand{\bopp}[1]{\mathcal{B}(#1)}
\newcommand{\boc}[2]{\mathcal{B}_{\infty}({#1},{#2})}
\newcommand{\bocc}[1]{\mathcal{B}_{\infty}(#1)}
 
\newcommand{\ham}[1]{\mathbb{#1}} 
 
\newcommand{\cs}{C^{\infty}} 
\newcommand{\ccs}{C_{0}^{\infty}} 
 
\newcommand{\tS}{\sc\mbox{S}\hspace{1.0pt}} 

\newcommand{\dtemp}{\tS^{\prime}} 

\newtheorem{theorem}{Theorem}[section]
\newtheorem{proposition}[theorem]{Proposition}
\newtheorem{lemma}[theorem]{Lemma}
\newtheorem{corollary}[theorem]{Corollary}

\theoremstyle{definition}
\newtheorem{definition}[theorem]{Definition}
\newtheorem{assumption}[theorem]{Assumption}

\theoremstyle{remark}
\newtheorem{remark}[theorem]{Remark}
\newtheorem{example}[theorem]{Example}

\begin{document}

\title{Complex absorbing potential method \\ 
        for the perturbed Dirac operator \footnote{First version 12 february 2012; second version 23 May 2012 (ArXiv version)}}

\author{J. Kungsman \\ 
             Department of Mathematics \\
             Uppsala University \\
             SE-751 06 Uppsala, Sweden
             \and 
             M. Melgaard \thanks{The second author acknowledges support by Science Foundation Ireland} \\
             Department of Mathematics \\ 
             School of Mathematical and Physical Sciences \\
             University of Sussex      \\
             Brighton BN1 9QH        \\ 
             Great Britain}                                               
\date{October 7, 2013}

 
\maketitle 
 
\begin{abstract}
The Complex Absorbing Potential (CAP) method is widely used to compute resonances in Quantum Chemistry, both for 
nonrelativistic and relativistic Hamiltonians. In the semiclassical limit $\hbar \to 0$ we consider resonances 
near the real axis and we establish the CAP method rigorously for the perturbed Dirac operator by proving that
individual resonances are perturbed eigenvalues of the nonselfadjoint CAP Hamiltonian, and vice versa. The proofs are 
based on pseudodifferential operator theory and microlocal analysis. 
\end{abstract}

\newpage

\tableofcontents

\newpage

\section{Introduction}
\label{jkmm2:intro}
One of the most successful methods for computing resonances in Quantum Chemistry is the Complex
Absorbing Potential (CAP) method, partly because it yields good approximations to the true resonances and, 
partly, because it is easy to implement numerically (see, e.g., Muga \textit{et al.} \cite{muga04}).

Within the semiclassical limit, i.e., as Planck's ``constant" $\hbar$ tends to zero, we study the CAP method 
rigorously when the governing Hamiltonian is a semiclassical Dirac operator
\begin{equation*}
  \ham{D} = -i c\hbar \sum _{j=1}^3 \a _j \pd_{x_{j}}  + \b mc^2 + \ham{V}(x), 
\label{jkmm2:intro-eq1}
\end{equation*}
acting on $\Lb^2(\R ^3 ; \C ^4) = \bigoplus _{j=1}^4 \Lb^2(\R ^3) =: (\Lb^2(\R ^3))^4$. Here the $\{ \a _j \}_{j=1}^3$ and $\b = \a _4$ are $4\times 4$ 
Dirac matrices obeying the anti-commutation relations 
\begin{equation*}
  \a _j \a _k + \a _k \a _j = 2 \d _{jk}\mbf{I}_{4}, \quad 1\le j,k \le 4,
\label{jkmm2:intro-eq2}
\end{equation*}
where $\mbf{I}_{n}$ is the $n\times n$ identity matrix. The potential $\ham{V}$ is assumed to have compact support. 

We define resonances through the method of complex distortion which has been widely applied in the context of 
Schr\"odinger operators but which subsequently was carried over to Dirac resonances in \cite{seba88}. Thus the 
resonances $z(\hbar )=E(\hbar ) \pm \G (\hbar ) /2$ appear as eigenvalues of a non-selfadjoint operator $\ham{D}_{\th}$ 
associated with $\ham{D}$. In applications one is interested in computing the resonance energy $E$ and the width $\G $, 
which is the inverse of the life-time of the corresponding resonant state. One way to do so is the CAP method, i.e., to 
augment the Hamiltonian by an imaginary potential and consider eigenvalues of the perturbed Hamiltonian as good 
approximations of the true resonances. In this paper we justify this method in the semiclassical approximation for 
resonances with $\G (\hbar ) = \Oc (\hbar ^N)$, $N\gg 1$, and show that such resonances give rise to eigenvalues 
of the CAP Hamiltonian $\ham{J}:=\ham{D}-i \ham{W}$ within distance at most $\hbar ^{-5}\log (\hbar ^{-1}) \G (\hbar ) + \Oc (\hbar ^\infty )$. 
Also the converse implication is proved. Both of these results hold under the assumption that the CAP is zero in the interaction
region, i.e. the support of the potential $\ham{V}$, and ``switched on" outside this region. In numerical implementations, 
however, the ``switch-on" point is moved inward towards the interaction region as much as possible to minimize the number 
of grid points used. If the classical Hamiltonian vector fields generated by the eigenvalues of the 
principal symbol of $\ham{D}$ are nontrapping (see Definition~\ref{jkmm2:hflow-def}), one can allow the supports to intersect 
which at worst increases the error by a factor $\hbar ^{-1}$. This requires the use of an Egorov type 
theorem for matrix valued Hamiltonians, which enables one to express the time evolution of quantum observables (self-adjoint
operators) in the semiclassical limit in terms of a classical dynamics of principal (matrix) symbols. The mentioned results deal 
with single resonances/eigenvalues and give no information regarding multiplicities; clusters of resonances will be treated in 
a future work.

Despite its success in Physics and Chemistry, only few rigorous justifications of the method exist. For (nonrelativistic, scalar valued) 
Schr\"{o}dinger operators with compactly supported electric potentials, Stefanov \cite{stefanov05} was the first to establish results 
similar to the above-mentioned ones. In the ``non-intersecting" case, he starts from a resonance and then, by considering a cutoff
resonant state (see Section~\ref{jkmm2:prfind1}), he constructs a quasimode (see Section~\ref{jkmm2:quasi}) which generates a 
perturbed resonance. In the ``intersecting" case, the previous scheme of proof only applies after a refined microlocal analysis, 
involving a propagation-of-singularities argument. Recently Kungsman and Melgaard carried over Stefanov's results to matrix valued 
Schr\"{o}dinger operators \cite{jkmm10}. The matrix valued setting is more complicated, in particular, in the ``intersecting" case, where 
one has to begin by solving Heisenberg's equations of motion semiclassically. Then, by applying a localization result away from the 
semiclassical wavefront set, it is possible to investigate how singularities propagate in this situation. The Egorov type statement, which 
is part of the proof by Kungsman and Melgaard \cite{jkmm10} differs from the scalar case because one also needs to propagate the 
matrix degrees of freedom. To push through this scheme of proof for matrix valued Schr\"{o}dinger operators, it was necessary to impose 
an additional technical (and restrictive) assumption in \cite{jkmm10}. An interesting feature of the present work, for the perturbed 
Dirac operator (also a matrix structure), is that one can avoid such technicalities, thus obtaining more natural and better results, and 
the afore-mentioned scheme of proof (using cutoff resonant states, Egorov type result, propagation of singularity argument and quasimodes), 
developed in \cite{jkmm10}, can be carried through, using a ``full" version of the matrix valued Egorov type theorem, see 
Lemma~\ref{jkmm2:prfind3-lem1}. We interpret this as yet another evidence of the fact that Dirac's description of the electron is a 
better physical model. 

Other rigorous results on resonances for Dirac operators are found in \cite{parisse91,parisse92,balslev92,amour01,khochman07}.

\newpage

\section{Preliminaries}
\label{jkmm2:prelim}
\noindent
\textbf{Notation}. Throughout the paper we denote by $C$ (with or without indices) various positive constants whose 
precise value is of no importance and their values may change from line to line; the ``constants'' usually depend on various parameters but not on $\hbar$. 
For $x_0\in \R ^3$ and $R>0$ the notation 
\begin{equation*}
    B(x_0,R) = \{x\,:\, |x-x_0|<R \}
\label{jkmm2:notation-eq3}
\end{equation*}
means an open ball centered at $x_0$ having radius $R$. For $x \in \R^{3}$ we denote $\langle x \rangle := (1 + |x|^{2})^{1/2}$. 
For a complex number $\z \in \C \setminus [-\infty ,0)$, we denote by $\z^{\frac{1}{2}}$ its branch of the square root with 
positive real part. The set $D(\z,r)=\{ z \in \C : |z-\z| <r , \: \z \in \C, \: r >0 \}$ defines an open disk in $\C$ with center in $\z$ and radius $r$. 
Complex rectangles $\{z\in \C \,:\, l\le \re z \le r,\,b\le \im z \le t \}$ are written 
\begin{align}
[l,r] + i[b,t].
\label{jkmm2:notation-eq5}
\end{align}
We shall denote by $\mathrm{M}_{4}(\C)$ the set of all $4 \times 4$ matrices over $\C$, equipped with the operator norm denoted by 
$\| \cdot \|_{4 \times 4}$. We let $\Hc := \Lb^{2}(\R^{3}, \C^{4})$ be the space of (equivalence classes of) 
$\C^{4}$-valued functions $\mbf{u} = (u_1, u_2, u_3, u_4)^t$ on $\R^{3}$ endowed with the inner product 
$$\langle \mbf{u}, \mbf{v} \rangle = \sum _{j=1}^4 \int _{\R ^3} u_i\overline{v_i} \, dx$$
such that $\langle \mbf{u}, \mbf{u} \rangle =: \|\mbf{u} \|^2$ is finite. 
The space $C_{0}^{\infty}(\R^{3})$ consists of all infinitely differentiable functions on $\R^{3}$ with compact support. We let $D_{x_{j}}=-i \pd/ \pd x_{j}$ 
and $D^{\g}=D_{x_{1}}^{\g_{1}} D_{x_{2}}^{\g_{1}} D_{x_{3}}^{\g_{3}}$ with standard multi-index notation $\g =(\g_{1}, \g_{2}, \g_{3}) \in \N_{0}^{3}$. 
The semiclassical Sobolev space of order one is denoted by $\Hb^{1}(\R^{3},\C^{4})$ and is equipped with the norm
$$
\|\mbf{u}\|_{\Hb ^1}^2 = \sum _{j=1}^4 \int _{\R ^3} (|\hbar \nabla u_j|^2 + |u_j|^2)\, dx.
$$
Moreover, the Schwartz space of rapidly decreasing functions and its dual space of 
tempered distributions are denoted by $\tS(\R^{3},\C^{4})$ and $\dtemp(\R^{3},\C^{4})$, respectively. 
For $\h _1,\h _2 \in \cs_{0} (\R^{n},[0,1])$ we use $\h _1 \prec \h _2$ to indicate that $\h _2 =1$ in a neighborhood of $\supp \h _1$ (i.e., the support of $\h _1$). We always assume cut-off functions take their values in $[0,1]$.   
\newline

\noindent
\textbf{Operators}. If $\mbf{A}$ is an operator on $\Lb^{2}(\R^{3}, \C ^4)$ its domain is denoted $\dom(\mbf{A})$. The spectrum of $\mbf{A}$ is the disjoint union of the discrete and essential spectra of $\mbf{A}$ and is designated by $\spec(\mbf{A}) = \smd(\mbf{A})\cup \sme(\mbf{A})$. Moreover, 
its resolvent set is denoted by $\rho(\mbf{A})$ and its resolvent is $\mbf{R}(\zeta)=(\mbf{A}-\zeta)^{-1}$. The spaces of bounded and compact operators between Hilbert 
spaces $\CH_{1}$ and $\CH_{2}$ are denoted by $\bop{\CH_{1}}{\CH_{2}}$ and $\boc{\CH_{1}}{\CH_{2}}$, respectively. If $\CH:=\CH_{1}=\CH_{2}$ 
we use the notation $\bopp{\CH}$ and $\bocc{\CH}$, respectively. The commutator of two operators $\mbf{A}$ and $\mbf{B}$, when defined, is denoted $[\mbf{A},\mbf{B}]=\mbf{AB} - \mbf{BA}$.  
The number of eigenvalues or resonances (counting multiplicities) of $\mbf{A}$ on a set $\O \subset \C$ will be denoted $\coun (\mbf{A},\O)$. 
Scalar-valued, respectively matrix-valued, operators are denoted by capitals, respectively boldface capitals, e.g. $\mbf{\h } = \h \mbf{I}_4$. If $\mbf{A} \in \Bc_{\infty}(\Hc)$ and if for some orthonormal 
basis $\{\mbf{f}_j\}$ of $\Hc $ the sum 
\begin{align}
\sum _{j} \langle (\mbf{A}^\ast \mbf{A})^{1/2}\mbf{f}_j, \mbf{f}_j \rangle
\label{jkmm2:ope-eq1} 
\end{align}
is finite, then this property is independent of the choice of orthonormal basis and we say that $\mbf{A}$ is of trace class, in symbols $\mbf{A} \in \Bc_{1}(\Hc)$, and the trace norm 
$\| \mbf{A} \| _{\Bc_{1}}$ is given by (\ref{jkmm2:ope-eq1}). Equivalently $\mbf{A}\in \Bc _1$ if and only if the sequence $\m _1 (\mbf{A}) \ge \m _2 (\mbf{A})\ge \cdots $ of eigenvalues of $(\mbf{A}^\ast \mbf{A})^{1/2}$, called singular values of $\mbf{A}$, is summable. The singular values satisfy Ky Fan's inequalities
\begin{align}
\m _{i+j-1}(\mbf{A} + \mbf{B}) &\le \m _i (\mbf{A}) + \m _j (\mbf{B}) \label{singular-values-compact-sum}\\
\m _{i+j-1}(\mbf{A}\mbf{B}) &\le \m _i (\mbf{A})\m _j (\mbf{B})  \label{singular-values-compact-product}
\end{align}
for $i,j\ge 0$ and $\mbf{A},\mbf{B} \in \Bc _\infty $ and also 
\begin{eqnarray}
\m _j (\mbf{A}\mbf{B}) &\le \|\mbf{B} \| \m _j (\mbf{A}) \label{singular-values-compact-bounded-1}\\
\m _j (\mbf{B}\mbf{A}) &\le \|\mbf{B} \| \m _j (\mbf{A}) \label{singular-values-compact-bounded-2}
\end{eqnarray}
whenever $\mbf{A}\in  \Bc _\infty $ and $\mbf{B}\in \Bc $. 
When $\mbf{A}$ is of trace class it is possible to extend the relation 
\begin{equation*}
\det (\mbf{1} - \mbf{A}) = \prod _{j} (1 - \l _j)
\label{jkmm2:ope-eq2} 
\end{equation*}
for $\mbf{A}$ of finite rank, where $\l _j$ are the eigenvalues of $\mbf{A}$, repeated according to multiplicity, so that $\det (\mbf{1}-\mbf{A}) \ne 0$ if and only if $\mbf{1}-\mbf{A}$ is invertible and 
\begin{equation}
    \det (\mbf{1} - \mbf{A}) \le e^{\|\lmbf{A} \| _{\Bc_{1}}} = e^{\sum \m _j (\lmbf{A})}. 
\label{jkmm2:ope-eq3} 
\end{equation}      
holds for any $\mbf{A}$ of trace class; see, e.g., \cite{sjo02} for details. 
\newline

\noindent
\textbf{Pseudodifferential operators}. For the (trivial) cotangent bundle of $\R ^3$ we write $\Tsf^\ast \R ^3$ and it is sometimes convenient to 
think of it as the product of space and frequency, i.e. $\Tsf^\ast \R ^3 = \R _x ^3 \times \R _\x ^3$. Let $m: \Tsf^{\ast }\R ^3 \to \R _+$ be a so called order function, i.e. a smooth function such that there are $C, N>0$ so that 
\begin{equation*}
m(x,\x )\le C \big (1 + (x-y)^2 + (\x - \y)^2 \big )^{N/2} m(y,\y)
\label{jkmm2:pseudo-eq1}
\end{equation*} 
for all $(x,\x), (y,\y) \in \Tsf^{\ast }\R ^3$. Then we define $\Ssf(m) \subset \cs{(\Tsf^{\ast }\R ^3)}\otimes \mathrm{M}_4(\C )$ to consist of all  $\mbf{a} \in \cs{(\Tsf^{\ast }\R ^3)}\otimes \mathrm{M}_{4}(\C )$ such that for all multi-indices $\a , \b \in \N _0 ^3$ there are constants $C_{\a ,\b} > 0$ with 
\begin{equation*}
\| \partial _\x ^{\a } \partial _x ^{\b } \mbf{a}(x,\x )\|_{4 \times 4} \le C_{\a ,\b}m(x,\x) \quad \text{for all } (x,\x)\in \Tsf^{\ast } \R ^3.  
\label{jkmm2:pseudo-eq2}
\end{equation*}
For $\mbf{a} \in \Ssf(m)$ we can define a corresponding Weyl quantization $\mbf{A}=\opw{\mbf{a}}$ on $\Lb^2(\R^{3}, \C^{4})$ by 
\begin{equation*}
(\mbf{A} \mbf{u})(x) = \frac{1}{(2\pi \hbar )^{3}} \iint \limits _{\Tsf^\ast \R ^3} e^{i (x-y)\cdot \x /\hbar } \mbf{a} \Big (\frac{x+y}{2}, \x\Big ) \mbf{u}(y) \,dy \, d\x. 
\label{jkmm2:pseudo-eq3}
\end{equation*}
For symbols that are bounded with all their derivatives we have the celebrated result by Calderon-Vaillancourt \cite[Theorem 7.11]{dimsjo99}.

\begin{proposition} 
Let $\mbf{a}\in \Ssf(1)$. Then $\opw{\mbf{a}}$ defines a continuous operator on $\Lb^{2}(\R^{3},\C^{4})$. 
\label{jkmm2:pseudo-cvprop}
\end{proposition}

We recall that if $m_1,\,m_2$ are order functions and $\mbf{a} \in \Ssf(m_1)$, $\mbf{b} \in \Ssf(m_2)$ then $m_1m_2$ is an order function and there exists $\mbf{a} \# \mbf{b} \in \Ssf(m_1m_2)$ so that 
\begin{align} 
\opw{\mbf{a}}\opw{\mbf{b}} = \opw{\mbf{a} \# \mbf{b}}.
\label{jkmm2:pseudo-eq4}
\end{align}
If for $\mbf{a} \in \Ssf(m)$ there are $\mbf{a}_j\in \Ssf(m)$ so that for any $N\in \N$ and $\a ,\b \in \N _0 ^3$ there exists $C_{N,\a }>0$ such that 
\begin{equation*}
\| \partial _\x ^\a \partial _x ^\b (\mbf{a} - \sum _{j=0}^{N-1} \hbar ^j \mbf{a}_j) \| \le C_{N ,\a} \hbar ^N m
\label{jkmm2:pesudo-eq5}
\end{equation*}
then we write $\mbf{a} \sim \sum _{j\ge 0}\hbar ^j \mbf{a}_j$ and we call $\mbf{a}_0$ and $\mbf{a}_1$ the principal, and subprincipal symbol of 
$\opw{\mbf{a}}$, respectively. The principal symbol of $\mbf{a} \# \mbf{b} $ in (\ref{jkmm2:pseudo-eq4}) is given by the product of the principal symbols of $\mbf{a}$ and $\mbf{b}$.

If $\mbf{a} \sim \sum \hbar ^j \mbf{a}_j$ and $\mbf{b} \sim \sum \hbar ^j \mbf{b}_j$ the symbol $\mbf{a} \# \mbf{b}$ has the asymptotic expansion   
\begin{align}
\mbf{a} \# \mbf{b} (x, \x) = \sum _{j, k, l \in \N _0} \frac{\hbar ^{j + k + l}}{j!} \Big ( \frac{i}{2}(\partial _x \partial _\eta - \partial _\x \partial _y) \Big ) ^j \mbf{a}_k (x, \x )\mbf{b}_l (y, \eta ) \Big | _{\substack{y=x \\ \eta = \x }}.
\label{product_expansion}
\end{align}
An operator $\mbf{A}= \opw{\mbf{a}}$ is called elliptic at $(x_0, \x _0)$ if 
$\mbf{a}^{-1}(x_0, \x _0)$ exists and belongs to $\Ssf(m^{-1})$. We say that \mbf{A} is elliptic if it is elliptic at every point. In the affirmative case, one can construct a parametrix $\mbf{q} \in \Ssf(m^{-1})$ which 
is an asymptotic inverse of $\mbf{a}$ in the sense of symbol products: 

\begin{lemma}
Suppose $\mbf{a} \in \Ssf(m)$ is elliptic in the sense that $\mbf{a}^{-1}(x,\xi)$ exists for all $(x,\xi) \in \Tsf^{\ast} \R^{d}$ and 
belongs to the class $\Ssf(m^{-1})$. Then there exists a parametrix $\mbf{q} \in \Ssf(m^{-1})$ with an asymptotic expansion 
of the form 
\begin{equation}
\mbf{q} \sim \mbf{a}^{-1} + \hbar (\mbf{a}^{-1} \# \mbf{r}) + \hbar^{2} ( \mbf{a}^{-1} \# \mbf{r} \# \mbf{r}) + \cdots
\label{jkmm2:pseudo-leminv-eq1}
\end{equation}
such that $\mbf{r} \in \Ssf(1)$ and
\begin{equation*}
\mbf{a} \# \mbf{q} \sim \mbf{q} \# \mbf{a} \sim \mbf{I}_{4}.
\end{equation*}
\label{jkmm2:pseudo-leminv}
\end{lemma}
For the proof of Theorem \ref{jkmm2:resultindthm2} it is convenient to have the following notion of microlocality: 
\begin{definition}
We say that $\mbf{u}\in \Lb ^2(\R^3 , \C ^4)$ is microlocally $\Oc (\ve (\hbar ))$ at $(x_0,\x _0)$  if there is $\mbf{a}\in \Ssf (1)$, invertible at $(x_0, \x _0)$, such that 
$$
\| \opw{\mbf{a}}\mbf{u} \| = \Oc (\ve (\hbar )), 
$$
uniformly as $\hbar \to 0$. 
\label{wavefront_def}
\end{definition}
\begin{lemma}
For $\mbf{a}$ as in Definition \ref{wavefront_def} we can find $\mbf{\h }_0 \in \Ssf (1)$ with support away from $(x_0, \x _0)$, such that $\mbf{a} + \mbf{\h }_0$ is everywhere invertible. 
\label{perturbation}
\end{lemma}
\begin{proof}
First assume $\mbf{a}(x_0,\xi _0) = \mbf{I}_4$. 
Let $\l _{\textrm{min}}(x, \x )$ be equal to the smallest of the eigenvalues of $\mbf{a}$ at $(x, \x)$. Then there is $\ve > 0$ so that $\l _{\textrm{min}}(x, \x ) > 1/2$ for all $(x, \x ) \in B((x_0, \x _0), \ve )$. Pick $M > \sup _{\R ^{2n}}\| \mbf{a} (x, \x ) \|$ and choose $\h _0$ to be a non-negative smooth function such that 
$$
\h _0(x, \xi) = \begin{cases}0 \quad &\text{in }B((x_0, \x _0), \frac{\ve }{2}), \\ M \quad  &\text{in }\mathbb{R}^{2n}\setminus B((x_0, \x _0), \ve ) \end{cases}. 
$$
Letting $\mbf{\h _0} = \h _0 \mbf{I}_4 $ it is clear that $\mbf{a} + \mbf{\h }_0$ is everywhere positive definite. 

In the general case we consider $\widetilde{\mbf{a}} (x,\xi ):= \mbf{a}(x, \xi ) \mbf{a}^{-1}(x_0, \xi _0)$. Then $\widetilde{\mbf{a}}$ satisfies $\widetilde{\mbf{a}}(x_0, \xi _0) = \mbf{I}_4$ so by the first part of the proof there is $\widetilde{\mbf{\h }}_0$, supported away from $(x_0, \xi _0)$, such that $\widetilde{\mbf{a}} + \widetilde{\mbf{\h }}_0$ is elliptic. Thus $\mbf{a}(x, \xi ) + \widetilde{\mbf{\h }}_0(x, \xi ) \mbf{a}(x_0, \xi _0)$ is everywhere invertible and $\mbf{\h }_0 (x, \x):=\widetilde{\mbf{\h }}_0(x, \xi ) \mbf{a}(x_0, \xi _0)$ belong to $\Ssf (1)$ and has support away from $(x_0, \x _0)$.  
\end{proof}
The following lemma shows the strength of Definition \ref{wavefront_def}. 
\begin{lemma}
Assume that $\mbf{u} \in \Lb ^2(\R ^3 , \C ^4)$ is microlocally $\Oc (\ve (\hbar ))$ at $(x_0, \x _0 )$. Then, for any $\mbf{b} \in \Ssf (1)$ with sufficiently small support near $(x_0, \x _0)$, it holds that 
$$
\| \opw{\mbf{b}}\mbf{u} \| = \Oc (\ve (\hbar ) + \hbar ^\infty ), 
$$  
uniformly as $\hbar \to 0$. 
\label{small_support}
\end{lemma}
\begin{proof}
Let $\mbf{\h }_0 $ be as in Lemma \ref{perturbation} and $\mbf{a}$ as in Definition \ref{wavefront_def}. Then we can find a parametrix $\mbf{q} \in \Ssf (1)$ with 
$$
\opw {\mbf{q}} \opw {\mbf{a} + \mbf{\h }_0} = \mbf{1} + \mbf{R},
$$
where $\|\mbf{R} \| = \Oc (\hbar ^\infty )$. Therefore, for any $\mbf{b} \in \Ssf (1)$, 
\begin{align*}
\opw {\mbf{b}} \mbf{u} = \opw {\mbf{b}} \opw {\mbf{q}} \opw {\mbf{a} + \mbf{\h }_0}\mbf{u} - \opw {\mbf{b}}\mbf{R}\mbf{u}.
\end{align*}
Using $\| \opw {\mbf{a}}\mbf{u} \| = \Oc (\ve (\hbar ))$ we obtain 
$$
\opw {\mbf{b}} \mbf{u} = \opw {\mbf{b} \# \mbf{q} \# \mbf{\h }_0}\mbf{u} + \Oc (\ve (\hbar )).  
$$
Since $\mbf{\h }_0$ has no support in a neighborhood of $(x_0 , \x _0)$ (see Lemma \ref{perturbation}) we can choose $\mbf{b}$ with sufficiently small support around $(x_0 , \x _0)$ so that $\supp (\mbf{b}) \cap \supp (\mbf{\h }_0) = \emptyset $. The lemma now follows from \eqref{product_expansion} and Proposition \ref{jkmm2:pseudo-cvprop}.  
\end{proof}

\section{Dirac and CAP Hamiltonians}
\label{jkmm2:diracope}
We introduce various assumptions and we define perturbed Dirac operators. Moreover, we introduce the CAP Hamiltonians.
\newline

\noindent
\textbf{The free Dirac operator}. The free semiclassical Dirac operator, describing the motion of a relativistic electron or positron 
without external forces, is the unique self-adjoint extension of the symmetric operator defined on $C_{0}^{\infty}(\R^{3}, \C^{4})$ 
in the Hilbert space $\Hc=\Lb^{2}(\R^{3}, \C^{4})$ by 
\begin{equation*}
\ham{D}_{0} := c \mbf{\a} \cdot \frac{\hbar}{i} \nabla + \b mc^{2} = -ic\hbar \sum _{j=1}^3 \a _j \pdf{}{x_j} + \b mc^2 , 
\label{jkmm2:freed-eq1}
\end{equation*}
where $\nabla=(\pd_{x_{1}}, \pd_{x_{2}}, \pd_{x_{3}})$ is the gradient, $c$ the speed of light, $m$ the electron mass, 
$\hbar$ the semiclassical parameter, and $\mbf{\a } := (\a _1, \a _2, \a _3)$ with $\a_{1}$, $\a_{2}$, $\a_{3}$, $\b$ being Hermitian 
$4 \times 4$ matrices, which satisfy the anti-commutation relations 
\begin{align*}
\begin{cases}
\a _i \a _j + \a _j \a _i = 2\d _{ij} \mbf{I}_{4}, \quad & \text{for } i,j=1,2,3, \\
\a _i \b + \b \a_i = 0, \quad &\text{for }i=1,2,3,
\end{cases}
\label{jkmm2:freed-eq2}
\end{align*}     
and $\b ^2 = \mbf{I}_{4}$.  For instance, one can use the ``standard representation'' 
\begin{equation*}
\a _i = \begin{pmatrix}0 &\s _i \\ \s _i &0 \end{pmatrix}, \quad \b = \begin{pmatrix}\mbf{I}_2 &0 \\ 0 &-\mbf{I}_2, \end{pmatrix}
\label{jkmm2:freed-eq3}
\end{equation*}
where 
\begin{equation}
\s _1 = \begin{pmatrix}0 &1 \\ 1 &0 \end{pmatrix}, \quad \s _2 = \begin{pmatrix}0 &-i \\ i &0 \end{pmatrix},\quad 
\s _3 = \begin{pmatrix}1 &0 \\ 0 &-1 \end{pmatrix}
\label{jkmm2:freed-eq4}
\end{equation}
are $2\times 2$ Pauli matrices. It is well-known that the resulting self-adjoint operator $\ham{D}_{0}$ has domain 
$\dom(\ham{D}_{0})=\Hb^{1}(\R^{3},\C^{4})$ and $\spec (\ham{D}_{0}) = \sme(\ham{D}_{0})=(-\infty , -mc^2 ]\cup [mc^2 , \infty )$; 
see, e.g., \cite{thaller92}.
\newline

\noindent
\textbf{Perturbed Dirac operator}.
To describe the interaction of a particle with external fields we perturb $\ham{D}_{0}$ by a potential 
$\ham{V} \in \cs (\R ^3) \otimes \mathrm{M}_{4}(\C )$, viewed as a multiplication operator on $\Hc$. 

\begin{assumption}
Let the potential $\ham{V} \, : \, \R^3 \to \mathrm{M}_{4}(\C)$ be Hermitian, smooth for all $x\in \R ^3$, and compactly supported; the number 
$R_0'>0$ is chosen such that $\supp \ham{V} \subset B(0,R_0')$.
\label{jkmm2:perturd-assumV}  
\end{assumption}

Under Assumption \ref{jkmm2:perturd-assumV} it is well-known that $\ham{D}:=\ham{D}_0 + \ham{V}$ is self-adjoint on 
$\dom(\ham{D}_0)=\Hb^1(\R^{3}, \C^{4})$. Moreover it follows from Weyl's theorem that 
$\sme(\ham{D}) = \sme(\ham{D}_0)=\spec(\ham{D}_{0})$; see, e.g., \cite[Section 4.3]{thaller92}. Henceforth 
we shall emphasize the dependence of $\hbar$ in $\ham{D}$ by writing $\ham{D}(\hbar)$.
\newline

\noindent
\textbf{Hamiltonian flow}.  Let $\mbf{d}_{0}$ be the principal symbol of $\ham{D}(\hbar)$ and let its eigenvalues be denoted by $\l_{j}$, $j =1, \ldots , 4$. 

The Hamiltonian trajectories (or bicharacteristics), denoted by $(x_{j}(t), \x_{j}(t))=:\F_{j} ^t (x_0, \x _0)$, $j=1, \ldots ,4$, are defined as the solutions of Hamilton's equations 
\begin{align*}
\begin{cases} 
x_{j} '(t)     &=\nabla_\x \l _{j} (x_{j} (t), \x _{j} (t)) \\
\x _{j} ' (t)  &= -\nabla_x \l _{j} (x_{j} (t), \x _{j} (t))
\end{cases}, \qquad (x_{j} (0), \x _{j} (0)) = (x_0, \x _0).
\label{jkmm2:hflow-eq1}
\end{align*}

\noindent
\textbf{Nontrapping condition}. We introduce the following nontrapping condition for the Hamiltonian flow generated by the 
eigenvalues $\l_{j}(x,\x)$, $j=1,\ldots ,4$.

\begin{definition}
We say that an energy band $J\subset \R$ is nontrapping for $\ham{D}(\hbar)$ if for any $R>0$ there exists $T_R>0$ 
such that 
\begin{equation*}
|x_{j} (t)|>R \text{ for } \l _{j} (x_0, \x _0) \in J \text{ provided }|t|>T_R \: \mbox{ and } \: j=1,\ldots ,4 . 
\label{jkmm2:hflow-eq2}
\end{equation*}
\label{jkmm2:hflow-def}
\end{definition}  

\noindent
\textbf{Hyperbolicity condition}. To avoid the difficulty of energy level crossings in certain situations, we shall introduce the following assumption.
\begin{assumption}
Distinct eigenvalues are said to satisfy the hyperbolicity condition if 
\begin{equation*}
        |\l _j (x,\x) - \l _k (x,\x )| \ge C \langle \x \rangle \quad \text{for all }(x,\x)\in \Tsf^{\ast} \R ^3  
\end{equation*}
for some constant $C>0$.
\label{jkmm2:assum-hyp} 
\end{assumption} 
\begin{example}
To illustrate Assumption~\ref{jkmm2:assum-hyp} we consider the Dirac operator describing a particle of mass $m$ and charge $e$ subject 
to external time-independent electromagnetic fields $\mbf{E}(x)= -\nabla \f(x)$ and $\mbf{B}(x)=\nabla \times \mbf{A}(x)$: 
\begin{equation*}
\ham{D}_{\lmbf{A},\f}(\hbar)= c \mbf{\a} \cdot \left( \frac{\hbar}{i} \nabla - \frac{e}{c} \mbf{A}(x)\right) + \b m c^{2} + e \f(x).
\end{equation*}
The principal symbol of $\ham{D}_{\lmbf{A},\f}(\hbar)$ is 
\begin{equation}
\mbf{d}_{0,\lmbf{A},\f}(x,\x) = c \mbf{\a} \cdot \left(\x-\frac{e}{c} \mbf{A}(x)\right)  + \b m c^{2} + e \f(x).
\label{jkmm2:perturd-eq1}
\end{equation}
For any $(x,\x)$ the symbol $\mbf{d}_{0,\lmbf{A},\f}(x,\x)$ is a Hermitian $4 \times 4$ matrix with two doubly degenerated eigenvalues 
\begin{equation*}
\l^{\pm}(x,\x) =  e \f(x) \pm \sqrt{ \left( c \x -e \mbf{A}(x) \right)^{2} +m^{2} c^{4} }  
\end{equation*}
associated with projection matrices 
\begin{equation*}
\mbf{\l}_{0}^{\pm} (x,\x) = \frac{1}{2} \left( \mbf{I}_{4} \pm \frac{ \mbf{\a} \cdot (c \x -e \mbf{A}(x) )+ \b mc^{2}}{ \sqrt{ ( c \x - e \mbf{A}(x) )^{2}+m^{2} c^{4} } }  \right)
\label{jkmm2:perturd-eq2}
\end{equation*}
onto the respective eigenspaces in $\C^{4}$. Since $\mbf{A}$ and $\f $ satisfy Asusmption~\ref{jkmm2:perturd-assumV} we may choose
\begin{equation*}
m(x,\x) : =  \sqrt{ \left( c \x- e \mbf{A} (x) \right)^{2}+m^{2} c^{4} }  ,
\label{jkmm2:perturd-eq3}
\end{equation*}
as an order function for the symbol $\mbf{d}_{0,\lmbf{A} ,\f}$. In particular,
\begin{equation*}
| \l^{+}(x,\x)-\l^{-}(x,\x) | = 2  m(x,\x) ,
\label{jkmm2:perturd-eq4}
\end{equation*}
which shows that Assumption~\ref{jkmm2:assum-hyp} holds true.
\end{example}   

\noindent
Cordes \cite{cordes82} imposes a similar condition on the eigenvalues of the symbol of an operator in a strictly hyperbolic system, and 
Bolte-Glaser \cite[Theorem 3.2]{bg04} prove a semiclassical version of Egorov's theorem under Assumption~\ref{jkmm2:assum-hyp}.
\newline

\noindent
\textbf{Complex absorbing potential Hamiltonian}.  
\begin{assumption}
Suppose $W \in \Lb ^\infty (\R ^3 , \C)$ is smooth and let $\ham{W} = W \mbf{I}_{4}$ be the operator on 
$\Lb^{2}(\R^{3}, \C^{4})$ induced by multiplication. Suppose, moreover, that $W$ satisfy the following properties:
\begin{enumerate}
\item[(i)] $\re W \ge 0$;
\item[(ii)] There is an $R_1>0$ such that $\supp W \subset \{|x|\ge R_1 \}$;
\item[(iii)] For some $\d _0>0$ and $R_2>R_1$ we have $\re W \ge \d _0$ for $|x|>R_2$;
\item[(iv)] $|\im W| \le C \sqrt{\re W}$ for some constant $C$.
\end{enumerate}
\label{jkmm2:capham-assumonw}
\end{assumption}
By property (i) $-i\ham{W}$ contributes a negative imaginary term which is necessary in order for the CAP to be absorbing. Property (ii) means absorption of the wave packet takes place away from the interaction region. If, on the other hand, $\ham{D}$ is assumed to be nontrapping on $\supp W$ in the sense of Definition \ref{jkmm2:hflow-def} and satisfy the hyperbolicity condition in Assumption \ref{jkmm2:assum-hyp} this condition can be relaxed, see Theorem \ref{jkmm2:resultindthm2}. Property (iii) is a strengthening of property (i) required to prove that eigenvalues of the CAP Hamiltonian defined below implies the existence of resonances nearby. We also allow $W$ to have a non-zero imaginary part as long as it is dominated by the real part in the sense of property (iv).   

In particular, we see that $i\ham{W}$ is not Hermitian. We now define two CAP operators. First, 
\begin{equation*}
\ham{J}_{\infty}(\hbar ) := \ham{D}(\hbar ) -i \ham{W} (x) \quad  \mbox{æon } \quad \Hc.
\label{jkmm2:capham-eq2}
\end{equation*}
Second, given $R > R_{2}$ let $\Hc_{R}(\hbar)$ be the restriction of $\Hc$ to the ball $B(0,R)$ and let 
$\ham{D}_{R}(\hbar)$ be the Dirichlet realization of $\ham{D}(\hbar)$ there. Define 
\begin{equation*}
\ham{J}_{R}(\hbar ) := \ham{D}_{R}(\hbar ) -i \ham{W} (x),
\label{jkmm2:capham-eq3}
\end{equation*}
We see that both $\ham{J}_{\infty}(\hbar)$ and $\ham{J}_{R}(\hbar)$ are closed unbounded operators with 
\begin{equation*}
\dom(\ham{J}_{\infty}(\hbar)) = \dom(\ham{D}(\hbar)) \quad \mbox{ and } \quad \dom(\ham{J}_{R}(\hbar)) = \dom(\ham{D}_{R}(\hbar))
\label{jkmm2:capham-eq4}
\end{equation*}
Furthermore, since $\re W \geq 0$, we see that $\C_{+}$ is contained in their resolvent sets.

\begin{remark}
In Physics and Chemistry the function $W$ is usually chosen to be real-valued, but, as in \cite{stefanov05}, we have 
defined $W$ as a complex-valued function.
\end{remark}
\section{Complex distortion and resonances}
\label{jkmm2:cscaling}
In the spirit of Aguilar-Balslev-Combes theory of resonances we summarize the spectral deformation theory for the Dirac operator, following 
Hunziker's approach, and we define resonances. Basic facts are stated without proofs; we refer to \cite{hunziker86,hissig96,khochman07} for 
details.

\subsection{Complex distortion}
We perform complex distortion outside of $B(0,R_2)\cup B(0,R_0')$ and for this purpose we introduce a smooth vector field $g$ with the following 
properties.

\begin{assumption} 
Let $g:\R ^3 \to \R ^3$ be a smooth function which satisfies:
\newline
\noindent
(i) $g(x) = 0 \text{ for }|x| \le R_0 \text{ where }R_0 > \max (R_0', R_2);$ 
\newline
\noindent
(ii) $g(x) = x \text{ for }|x| > R_0 + \y \text{ for some }\y >0;$ 
\newline
\noindent
(iii) $\sup _{x\in \R ^3} \| (D g)(x) \| < \sqrt{2}$ with $(D g)(x)$ being the Jacobian matrix of $g$. 
\label{jkmm2:cscaling-gfunc}
\end{assumption}

The parameter $R_0$ will be chosen suitably in different circumstances. This will not affect the set of resonances we study.

Henceforth we impose Assumption~\ref{jkmm2:cscaling-gfunc}. For fixed $\ve \in (0,1)$ and 
\begin{equation*}
\th \in D_\ve := \Big \{ \th \in \C \,:\, |\th | < r_\ve :=\frac{\ve }{\sqrt{1 + \ve ^2}} \Big \},
\label{jkmm2:cscaling-eq2}
\end{equation*}
we let $\f _\th :\R ^3 \to \R ^3$ be defined by $\f _\th (x) = x + \th g(x)$ and we denote the Jacobian determinant of $\f_{\th}$ by $J_{\th}$.
We then define $\mbf{U}_{\th} \, : \, \mbf{\tS} (\R^{3}, \C^4) \to \mbf{\tS} (\R^{3}, \C^4) $ for $\th \in (-r_\ve , r_\ve )$ by 
\begin{equation*}
\mbf{U}_{\th} \mbf{f} (x) = J_{\th}^{1/2}(x) \mbf{f}(\f_{\th} (x)).
\label{jkmm2:cscaling-eq3}
\end{equation*}  
One has: 
\newline

\noindent
\textbf{(P1)}: \textit{The map $\mbf{U}_{\th}$ extends, for $\th \in (-r_\ve , r_\ve )$, to a unitary operator on $\Lb^{2}(\R^{3},\C^{4})$}.  
\newline

\begin{definition}
Let $\Ac $ be the linear space of all entire functions $\mbf{f}=(f_i)_{1\le i \le 4}$ such that for any $0<\ve <1$ and $k\in \N$ we have 
\begin{equation*}
\lim _{\substack{|z| \to \infty \\ z\in C_\ve }} |z|^k |f_i(z)| = 0 \quad \text{for }1\le i\le 4 ,
\end{equation*}
where 
\begin{equation}
C_{\ve} = \{ z\in \C ^3 \,:\, |\im z | \leq \ve |\re z|, \, |\re z|>\max (R_0', R_2) \}.
\label{jkmm2:cscaling-def1-eq1}
\end{equation}
\label{jkmm2:cscaling-def1}
\end{definition}

We now define the class of \textit{analytic vectors}:

\begin{definition}
Let $\Bc \subset \Lb^{2}(\R^{3},\C^{4})$ be the set of $\mbf{\p} \in \Lb^{2}(\R^{3},\C^{4})$ such that there exists $\mbf{f} \in \Ac$ 
with $\mbf{f}(x)=\mbf{\p} (x)$ for $x\in \R ^3$. 
\label{jkmm2:cscaling-def2} 
\end{definition}
\noindent
Then:
\newline

\noindent
\textbf{(P2)}: \textit{The set $\Bc $ is dense in $\Lb^{2}(\R^{3},\C^{4})$}. 
\newline
This statement follows from the fact that $\Bc $ is a linear space which contain the set of Hermite functions which has a dense span.  

Moreover, for $\Bc $ to be a set of analytic vectors for $\mbf{U}_{\th}$ (see e.g. \cite{hissig96}), we need the following fact; wherein we 
allow $\th$ to become non-real. 
\newline

\noindent
\textbf{(P3)}: \textit{For all $\th \in D_\ve $ we have}
\begin{itemize}
\item[(i)] \textit{For all $\mbf{f} \in \Bc $ the map $\th \mapsto \mbf{U}_{\th} \mbf{f}$ is analytic}. 
\item[(ii)] \textit{$\mbf{U}_{\th} \Bc $ is dense in $\Lb^{2}(\R^{3},\C^{4})$. }
\end{itemize}

\noindent
We are now ready to define the family of spectrally deformed Dirac operators.

\begin{definition}
For $\th \in D_\ve ^+:= D_\ve \cap \{ {\rm Im}\, z \geq 0 \}$ we let
\begin{equation*}
\ham{D}_{\th}(\hbar) := \mbf{U}_{\th} \ham{D}(\hbar) \mbf{U}_{\th }^{-1} = \mbf{U}_{\th} \ham{D}_{0}(\hbar) \mbf{U}_{\th }^{-1} + \mbf{U}_{\th} \ham{V} \mbf{U}_{\th }^{-1} =:\ham{D}_{0,\th }(\hbar) + \ham{V}( \f _{\th}  (x)).
\label{jkmm2:cscaling-def3}
\end{equation*}
\end{definition}  

\noindent
We have: 
\newline

\noindent
\textbf{(P4)}: 
\textit{For $\th _0 \in D_\ve ^+$ the eigenvalues of $\ham{D}_{\th_0}(\hbar)$ are independent of the spectral deformation family $\{ \mbf{U}_{\th_{0}} \}$.}
\newline

\noindent
Khochman \cite[Lemma~3]{khochman07} proves the following representation of the free deformed Hamiltonian $\ham{D}_{0,\th }(\hbar) = \mbf{U}_{\th} \ham{D}_{0}(\hbar) \mbf{U}_{\th}^{-1}$.
\begin{lemma}
For $\th \in D_\ve $
\begin{equation*}
\ham{D}_{0,\th } (\hbar)= -\frac{1}{1+\th} ic \hbar  \mbf{\a }\cdot \nabla+ \b mc^2 + \mbf{Q}_{\th} (x, \hbar \partial _{x_j}), 
\end{equation*}
where $\mbf{Q}_{\th} (x, \hbar \partial _{x_j}) = \sum _{|\g | \le 1} \mbf{a}_{\g} (x,\th ) (\hbar \partial _{x_j})^{\g} $ with $\mbf{a}_{\g} (x, \cdot )$ analytic and $\Oc(\th)$,  and $\mbf{a}_{\g} (\cdot , \th ) \in C_{0}^{\infty}(B(0,R_0 + 2\eta ),\C^{4})$.
\end{lemma}

\begin{proof}
Since
\begin{equation*}
  \ham{D}_{\th , 0} = \mbf{U}_{\th} \ham{D}_0 \mbf{U}_{\th }^{-1} = -ic \hbar \sum _{j=1}^3 \a _j \mbf{U}_{\th}  
  \partial _j \mbf{U}_{\th }^{-1} + \b mc^2   
\end{equation*}
we need only compute, for any $\mbf{f} \in \mbf{\tS}(\R^{3},\C^{4})$,  
\begin{align*}
  \mbf{U}_{\th} \partial _j \mbf{U}_{\th}^{-1} \mbf{f}(x) &= J_{\th}^{1/2}(x) \Big (\partial _j J_{\th}^{-1/2} (\f_{\th}^{-1}(x)) \mbf{f}(\f_{\th}^{-1}(x))\Big ) (\f_{\th} (x)) \\ 
&= -\frac{1}{2}J_{\th}^{-1}(x) \partial _j J_{\th}(x) \mbf{f}(x) + \sum _{k=1}^3 \partial_{j} \f_{\th , k}^{-1}\big (\f _\th (x) \big ) \partial_k \mbf{f}(x)    
\end{align*}
where we use the notation $\f_{\th}^{-1} = (\f_{\th ,1}^{-1}, \f_{\th ,2}^{-1}, \f_{\th ,3}^{-1})$. By 
Assumption~\ref{jkmm2:cscaling-gfunc} we have $\partial_{j} \f_{\th ,k}^{-1} (\f_{\th} (x))= (1+\th)^{-1} \d _{jk}$, $k=1,2,3$, provided $|x|>R_0+\eta $. Thus, if $\h \in \ccs (B(0,R_0+2\eta ))$ is taken to equal $1$ near $B(0,R_0 + \eta )$ we have
\begin{align*}
&\mbf{U}_{\th} \partial_{j} \mbf{U}_{\th}^{-1} \mbf{f}(x) = -\frac{1}{2} J_{\th}^{-1}(x) \partial_{j} J_{\th} (x) \mbf{f}(x) 
+ \frac{1}{1+\th }\partial_{j} \mbf{f}(x)(1 - \h ) \\
&\phantom{ooooooooooooooooooooooooooooooooooo}+ \sum _{k=1}^3 \partial_j  \f_{\th , k}^{-1} \big (\f_{\th} (x) \big ) \partial_{k} \mbf{f}(x)\h \\
&= \frac{1}{1+\th } \partial_{j} \mbf{f}(x) \\
&\phantom{oo}+ \Big \{-\frac{1}{2}J_{\th}^{-1}(x) \partial_{j} J_{\th} (x) \mbf{f}(x) - \frac{1}{1+\th } \partial_{j} \mbf{f}(x) \h + \sum _{k=1}^3 \partial_{j} \f_{\th , k}^{-1}\big (\f _\th (x) \big ) \partial_k \mbf{f}(x) \h \Big \}
\end{align*} 
where the terms in brackets are what makes $Q_{\th}$ after multiplication by $-ic\hbar \a _j $ and summation over 
$j=1,2,3$. 
\end{proof}

\begin{remark}
In particular we see that, for $\th \in D_\ve $, $\th \mapsto \ham{D}_{0,\th }$ is a holomorphic family of type (A) in the sense of Kato (see \cite[p. 375]{kato95}).
\end{remark}
\noindent
The above representation can be modified to the following more variable one:  
\begin{lemma}
For $\th \in D_\ve $ we have, using the principal branch of the cube root, 
\begin{align*}
\ham{D}_\th = -J_\th ^{-1/3}ic \hbar \mbf{\a }\cdot \nabla + \b mc^2 + \widetilde{\mbf{Q}}_\th(x,\hbar \partial _{x_j}),
\end{align*}
where $\widetilde{\mbf{Q}}_\th (x,\hbar \partial _{x_j})= \sum _{|\g | \le 1} \widetilde{\mbf{a}}_{\g} (x,\th ) (\hbar \partial _{x_j})^{\g}$ with the $\widetilde{\mbf{a}}_{\g}(\cdot , \th )$ supported in $\{R_0< |x| <R_0 + 2\eta \}$.  
\label{jkmm2:cscaling-lem3}  
\end{lemma}
\noindent
Using Lemma~\ref{jkmm2:cscaling-lem3} we are now ready to show the following:
\begin{proposition}
For $\theta \in D_{\ve }^+$, $\re z > mc^2$ and any $K\in \Z _+$ there is $C_K>0$ such that
$$\|(\ham{D}_\theta - z)^{-1}\| \le \frac{C_K}{\im z} \quad \text{for }\im z > \hbar ^K,$$
provided $\hbar $ is small enough. 
\label{jkmm2:cscaled-resolv-upper}
\end{proposition}
\begin{proof}
We prove the result by studying the quantity
\begin{equation}
  \im \langle J_\th ^{1/3} ( \ham{D}_ \th - \re z) \mbf{u}, \mbf{u} \rangle .
\label{jkmm2:clusprf-lem1-eq2}
\end{equation}
Take $\h \in \ccs (B(0,R_0))$ which equals $1$ near $B(0,R_0')$. Then, using the fact that $J_{\th}(x) = 1$ 
and $\ham{D}_{\th} = \ham{D}$ for $|x|<R_0$,    
\begin{align*}
& \im \langle J_{\th}^{1/3} ( \ham{D}_{\th} - \re z ) \mbf{u}, \mbf{u} \rangle = \im \langle (\ham{D} - \re z ) \mbf{\h} \mbf{u}, \mbf{\h} \mbf{u} \rangle  \\
& +\im \langle ( \ham{D} - \re z ) \mbf{\h} \mbf{u}, (\mbf{1}-\mbf{\h} ) \mbf{u} \rangle + \im \langle ( \ham{D} - \re z )(\mbf{1}-\mbf{\h} ) \mbf{u}, \mbf{\h}  \mbf{u} \rangle \\
& +\im \langle J_{\th}^{1/3}( \ham{D}_{\th} - \re z )(\mbf{1}- \mbf{\h} ) \mbf{u} , (\mbf{1}- \mbf{\h} )  \mbf{u} \rangle \\
& =\im \langle J_{\th}^{1/3}( \ham{D}_{\th} - \re z )(\mbf{1}- \mbf{\h} ) \mbf{u} , (\mbf{1}- \mbf{\h} )  \mbf{u} \rangle , 
\label{jkmm2:clusprf-lem1-eq3}
\end{align*} 
because $\ham{D}$ is symmetric. Consequently, in order to estimate 
$\im \langle J_{\th}^{1/3}( \ham{D}_{\th} - \re z ) \mbf{u}, \mbf{u} \rangle $, it suffices to consider $\mbf{u}$ with support in 
$\R^{3} \setminus B(0,R_0')$, meaning we can replace $\ham{D}_\th $ by $\ham{D}_{0,\th } $.

By Lemma \ref{jkmm2:cscaling-lem3} we may write 
$$\ham{D}_{0,\th } = -J_\th ^{-1/3}ic\hbar \mbf{\a }\cdot \nabla + \b mc^2 + \widetilde{\mbf{Q}}_\th  $$
where $ \widetilde{\mbf{Q}}_\th$ is a first order differential operator having smooth coefficients supported in some subset $U$ of $B(0,R_0 + 2\eta ) \setminus B(0,R_0)$. Thus \eqref{jkmm2:clusprf-lem1-eq2} becomes
\begin{align*}
\im \langle J_\th ^{1/3} ( \ham{D}_ \th - \re z) \mbf{u}, \mbf{u} \rangle &= mc^2 \langle \im (J_\th ^{1/3})\b \mbf{u}, \mbf{u} \rangle 
+ \im \langle J_\th ^{1/3}\widetilde{\mbf{Q}}_\th \mbf{u}, \mbf{u} \rangle \\ 
&\phantom{ooooooooooooooooooooo}- \re z \langle \im (J_\th ^{1/3})\mbf{u}, \mbf{u} \rangle \\
&\le (mc^2 - \re z)\langle \im (J_\th ^{1/3}) \mbf{u}, \mbf{u} \rangle + C\|\mbf{u} \|_{\Hb ^1(U)} \|\mbf{u}\|.
\end{align*}
Now, the first term on the right is non-positive. Moreover,
\begin{align*}
\|\mbf{u} \|_{\Hb ^1(U )} &= \|(\mbf{1} - \mbf{\h })\mbf{u} \|_{\Hb ^1(U)} \le C_1\| (\ham{D}_{0,\theta } - z)(\mbf{1} - \mbf{\h })\mbf{u} \| = C_1\| (\ham{D}_{\theta } - z)(\mbf{1} - \mbf{\h })\mbf{u} \| \\
&\le C_1(\| (\ham{D}_\theta  - z)\mbf{u} \| + \|[\ham{D}_{0,\theta },  \mbf{\h }] \mbf{u}\|) \\
&\le C_1(\| (\ham{D}_\theta  - z)\mbf{u} \| +\hbar  \|\mbf{u}\|_{\supp (\nabla \h)} )  
\end{align*}
since $z \not \in \spec  (\ham{D}_{0,\theta })$ (see Sec. \ref{jkmm2:restate}). Next take $\h _2$ having the same properties as $\h $ but also $\h _2 \prec \h $. Then, in the same way, 
\begin{align*}
\|\mbf{u}\|_{\supp (\nabla \h)} =  \|(\mbf{1} - \mbf{\h}_2)\mbf{u}\|_{\supp (\nabla \h)} \le C_2(\| (\ham{D}_\theta  - z)\mbf{u} \| +\hbar  \|\mbf{u}\|_{\supp (\nabla \h _2)} ). 
\end{align*}
Continuing in this way, with $\h _K \prec \cdots \prec \h _2 \prec \h $, we obtain, for any $K\in \Z _+$,
\begin{equation}
\|\mbf{u} \|_{\Hb ^1(U)} \le C_K(\| (\ham{D}_\theta  - z)\mbf{u} \| +\hbar ^K \|\mbf{u}\|_{\supp (\nabla \h _K)} ), \quad \text{for }z \not \in \spec  (\ham{D}_{0,\theta }) .   
\label{jkmm2:ellipticDtheta}
\end{equation} 
Therefore 
\begin{align*}
\im \langle J_\th ^{1/3} ( \ham{D}_ \th - \re z) \mbf{u}, \mbf{u} \rangle \le C_K (\| (\ham{D}_\theta  - z)\mbf{u}\| + \hbar ^K \|\mbf{u}\|) \|\|\mbf{u} \| 
\end {align*}  
which together with 
\begin{align*}\im \langle J_\th ^{1/3} ( \ham{D}_ \th - z) \mbf{u}, \mbf{u} \rangle = \im \langle J_\th ^{1/3} ( \ham{D}_ \th - \re z) \mbf{u}, \mbf{u} \rangle  
-(\im z) \langle \re (J_\th ^{1/3})\mbf{u}, \mbf{u} \rangle 
\end{align*}
gives 
$$C \| (\ham{D}_ \th - z)\mbf{u} \| \ge (C_0\im z - C_{K+1}\hbar ^{K+1})\|\mbf{u}\| \ge \frac{C_0}{2}\im z \|u \|, \quad \im z > \hbar ^K, $$
provided $\hbar $ is sufficiently small. This proves the lemma. 
\end{proof}   
%
%
%
\subsection{Resonances}
\label{jkmm2:restate}
Let
\begin{equation*}
  \Sigma _\th := \left\{ \, z \in \C \, : \, z = \pm c\Big ( \frac{\l }{(1+\th )^2}+m^2c^2 \Big )^{1/2},\,\l \in {\rm [} 0, \infty {\rm )} \,  \right\},
\label{jkmm2:restate-eq1}
\end{equation*}
where we have taken the principal branch of the square root function, and put (see Figure~\ref{fig1}) 
\begin{equation*}
  S_{\th _0} := \bigcup _{\th \in D_{\ve , \th _0}^+} \Sigma _\th 
\label{jkmm2:restate-eq2}
\end{equation*} 
where
\begin{equation*}
 D_{\ve , \th _0}^+ := \left\{ \, \th \in D_{\ve }^+\,:\, \arg (1 + \th) < \arg (1 + \th_{0} ),\; \frac{1}{|1 + \th |} < \frac{1}{|1 + \th_{0} |} \, \right\}.
\label{jkmm2:restate-eq3}
\end{equation*}

\noindent
We have the following results, where the second asserts that the essential spectrum of $\ham{D}_{0, \th}(\hbar)$ is invariant under the influence 
of a potential satisfying Assumption \ref{jkmm2:perturd-assumV}.
\newline

\noindent
\textbf{(P5)}: $\sme(\ham{D}_{0,\th}(\hbar)) = \Sigma _{\th}$.
\newline

\noindent
\textbf{(P6)}: $\sme(\ham{D}_{\th}(\hbar) ) = \Sigma _\th$ .
\newline  

\noindent
In view of Property~\textbf{(P4)} the following definition makes sense.

\begin{definition}
The set of resonances of $\ham{D}(\hbar)$ in $S_{\th _0}\cup \R$, designated $\Res(\ham{D}(\hbar))$ 
(with $\th _0$ suppressed), is the set of eigenvalues of $\ham{D}_{\th _0}(\hbar)$. If $z_0$ is a 
resonance, then the spectral (or Riesz) projection
\begin{equation*}
\mbf{\Pi}_{z_0} = \frac{1}{2\pi i} \oint \limits _{|z-z_0|\ll 1} (\ham{D}_{\th}(\hbar) - z)^{-1}\,dz
\label{jkmm2:restate-riesz}
\end{equation*}
makes sense and has finite rank. We define the multiplicity of $z_0$ to be the rank of $\mbf{\Pi}_{z_0}$.
\label{jkmm2:restate-def1}
\end{definition}

We will restrict ourselves to the study of resonances having positive energies. Namely, we assume that the 
resonances are located in a rectangle $\Rc$ satisfying the following: 

\begin{assumption}
We say that a complex rectangle $\Rc$ as in (\ref{jkmm2:notation-eq5}) satisfies the assumption 
$(\textbf{A}_{\Rc }^+)$ if $l>mc^2$, $b<0<t$ and there exists $\th _0\in D_\ve ^+$ such that 
$\Rc \cap \Sigma _{\th _0} = \emptyset $.   (cf. Figure~\ref{fig1}).  
\label{jkmm2:restate-rectangle}
\end{assumption}

\noindent
In Figure~\ref{fig1} we show a typical scenario when we fix a $\th _0 \in D_\ve ^+$ to uncover the resonances in $S_{\th _0}$.
\begin{figure}[hbt!]
\begin{center}
\def\svgwidth{0.8\columnwidth}
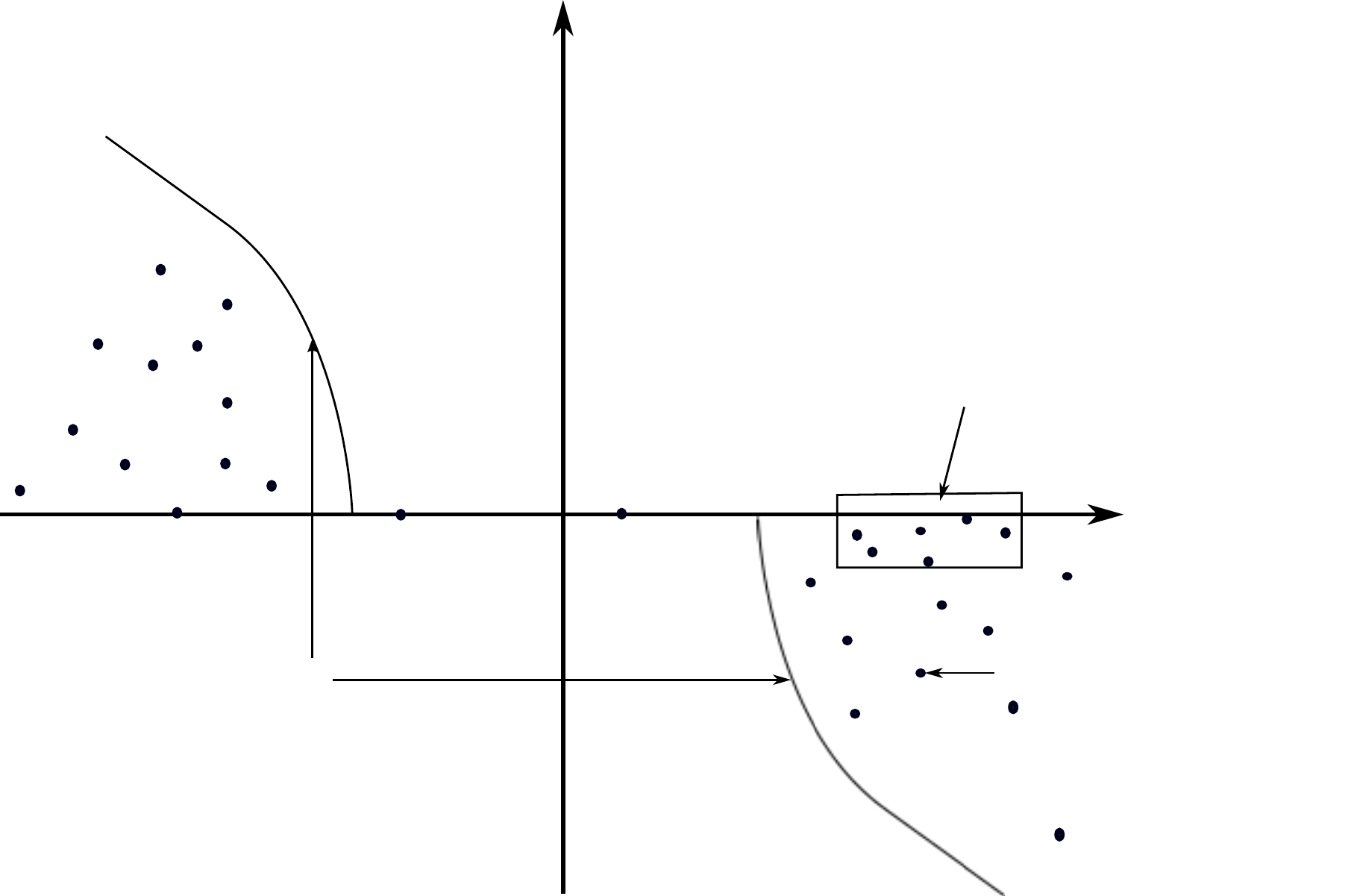
\end{center}
\caption{The set $S_{\th _0}$ and a rectangle $\Rc $ satisfying $(\textbf{A}_{\Rc }^+)$.}
\label{fig1}
\end{figure}

\noindent
The following upper bound on the number of resonances, not necessarily close to the real axis, will be used repeatedly throughout the paper.  
A proof can be found in Khochman \cite{khochman07}, who follows Nedelec's work on matrix valued Schr\"{o}dinger operators 
\cite{nedelec01} (in turn inspired by Sj\"{o}strand \cite{sjo97b}).

\begin{theorem}
Let $\ham{V}$ satisfy Assumption~\ref{jkmm2:perturd-assumV} and let $\Rc $ be a complex rectangle satisfying Assumption 
$(\textbf{A}_{\Rc }^+)$. Then 
\begin{equation*}
    \coun \left( \ham{D}(\hbar), \Res (\ham{D}(\hbar)) \cap \Rc \right) \leq C (\Rc ) \hbar ^{-3}.
\end{equation*} 
\label{jkmm2:restate-thm1}
\end{theorem}

We will need the following important \textit{a priori} resolvent estimate for $\ham{D}_{\th}(\hbar)$ away from the critical set, which is useful for 
applying the semiclassical maximum principle (see, e.g., \cite{tangzwor98} or \cite[Corollary 1]{stefanov05}). Due to lack of space, we omit 
its lengthy proof (which is based on ideas from Sj\"{o}strand and Zworski \cite{sjozwor91} and Sj\"{o}strand \cite{sjo97b}). 

\begin{proposition}
Let Assumption~\ref{jkmm2:perturd-assumV} hold. Let $\Rc$ be a complex rectangle satisfying Assumption $(\textbf{A}_{\Rc }^+)$ and 
assume $g:(0,\hbar _0] \to \R _+$ is $o(1)$. Then there are constants $A=A(\Rc )>0$ and $\hbar _1 \in (0,\hbar _0)$ such that
\begin{align}
\| (\ham{D}_{\th} (\hbar) - z)^{-1}\| \le Ae^{A\hbar ^{-3}\log \frac{1}{g(\hbar )}} \quad \text{for all } 
        z\in \Rc \setminus \bigcup _{z_j \in \Res (\ham{D}(\hbar)) \cap \Rc } D(z_j, g(\hbar )),
        \label{jkmm2:aux-prop2-eq1}
\end{align}
for all $0<\hbar \le \hbar _1$.
\label{jkmm2:aux-prop2}
\end{proposition}
\section{Main results}
Henceforth we always impose Assumption~\ref{jkmm2:perturd-assumV} and Assumption~\ref{jkmm2:capham-assumonw}. Moreover, 
$\ham{J}(\hbar)$ represents either $\ham{J}_{\infty}(\hbar)$ or $\ham{J}_{R}(\hbar)$. Throughout we shall assume that
$mc^2 < l_0 < r_0 < \infty $ (here $l_0$ and $r_0$ are independent of $\hbar $). 

\subsubsection{The case $R_{0}^{\prime} < R_{1}$}
\label{jkmm2:resultind1}
Bear in mind that $\supp \ham{W} \subset \R^{3} \setminus B(0,R_{1})$. We obtain the following result, which shows how 
a single resonance of $\ham{D}(\hbar )$ generates a single eigenvalue of $\ham{J}(\hbar )$ nearby, and vice versa. 

\begin{theorem} \hspace*{4cm} 
\newline
\noindent
1. Let $R_{0}^{\prime} < R_{1}$. Suppose $z_0(\hbar )$ is a resonance of $\ham{D}(\hbar)$ in 
\begin{equation*}
[l_0, r_0] + i\Big [- \frac{\hbar^{5}}{C\log \frac{1}{\hbar }}, 0 \Big ], \quad C \gg 1.
\label{jkmm2:boxinthm1}
\end{equation*}
Then there is an $\hbar_{0} \in {\rm (} 0,1 {\rm ]}$ such that, for $0 < \hbar \leq \hbar_{0}$, $\ham{J}(\hbar )$ has an eigenvalue in 
\begin{equation}
\big [\re z_0 (\hbar ) - \ve (\hbar ) \log \frac{1}{\hbar }, \re z_0 (\hbar ) + \ve (\hbar ) \log \frac{1}{\hbar } \big ] + i[-\ve (\hbar ), 0]
\label{jkmm2:narrowboxinthm1}
\end{equation}
where $\ve (\hbar ) = - \hbar ^{-5} \im z_0 (\hbar ) + \Oc (\hbar ^\infty )$.
\newline
\noindent
2. Let $R_{0}^{\prime} < R_{1}$. Suppose $w_0(\hbar )$ is an eigenvalue of $\ham{J}(\hbar)$ in 
\begin{equation*}
[l_0, r_0] + i\Big [- \Big (\frac{\hbar^{4}}{C\log \frac{1}{\hbar }} \Big )^2, 0 \Big ], \quad C \gg 1.
\end{equation*}
Then there is an $\hbar_{0} \in {\rm (} 0,1 {\rm ]}$ such that, for $0 < \hbar \leq \hbar_{0}$, $\ham{D}(\hbar )$ has a resonance in 
\eqref{jkmm2:narrowboxinthm1} with $\ve (\hbar ) =  \hbar ^{-4} \sqrt{-\im w_0 (\hbar )} + \Oc (\hbar ^\infty )$. 
\label{jkmm2:resultindthm1} 
\end{theorem}

\subsubsection{The case $R_{1} < R_{0}^{\prime}$}
As the following theorem shows we only worsen the error by at most a factor $\hbar ^{-1}$ if we allow the 
supports of $\ham{V}$ and $\ham{W}$ to intersect. To establish it we need to impose both the nontrapping assumption
and the hyperbolicity condition.

\begin{theorem} Let $R_{1} < R_{0}^{\prime}$.  Suppose that $\ham{D}(\hbar)$ is nontrapping for $|x|>R_1$ on the interval $J=[l_{0},r_{0}]$; in 
the sense of Definition~\ref{jkmm2:hflow-def}. Moreover, let Assumption~\ref{jkmm2:assum-hyp} be satisfied and suppose 
$z_0(\hbar )$ is a resonance of $\ham{D}(\hbar )$ in 
\begin{equation*}
[l_0, r_0] + i\Big [- \frac{\hbar^{6}}{C\log \frac{1}{\hbar }}, 0 \Big ], \quad C \gg 1.
\end{equation*}
Then there is an $\hbar_{0} \in {\rm (} 0,1 {\rm ]}$ such that, for $0 < \hbar \leq \hbar_{0}$, $\ham{J}(\hbar )$ has an eigenvalue in \eqref{jkmm2:narrowboxinthm1} with
\begin{equation*}
\ve(\hbar ) = -\hbar ^{-6} \im z_{0}(\hbar ) + \Oc (\hbar ^\infty ).
\end{equation*}
\label{jkmm2:resultindthm2}
\end{theorem}
\section{Properties of CAP Hamiltonians}
\label{jkmm2:aux}
Herein we study the spectral properties of the CAP Hamiltonians. We give an estimate of the number of eigenvalues of 
$\ham{J}(\hbar)$ on a rectangle. The result is an analogue of the estimate in Theorem~\ref{jkmm2:restate-thm1} for 
$\ham{D}(\hbar)$, however this time for the number of eigenvalues of 
$\ham{J}(\hbar)$ rather than the resonances of $\ham{D}(\hbar)$. Our approach is inspired by Stefanov \cite{stefanov05}.

Since the following result is independent of $\hbar $ we need not indicate that we have a family of $\hbar $-dependent operators. 
\begin{lemma}
The resolvent $(\ham{J} - z)^{-1}$ exists as a meromorphic operator in $\im z > -\d _0$ with the poles being the eigenvalues of finite multiplicity. 
\label{J-resolvent-meromorphic-delta-lemma}
\end{lemma}
\begin{proof}
Let $\h _1 + \h _2 + \h _3 = 1$ be a smooth partition of unity with $\h _1 = 1$ near $B(0,R_0')$ and supported in $B(0,\tfrac{R_0' + R_1}{2})$, $\h _2$ compactly supported and $\h _3$ supported in $|x|>R_2$. Let $\widetilde{\h }_j \succ \h _j$ have the same support properties. The fact that $\{\widetilde{\h }_j \} $ is not a partition of unity does not matter. Define $W_1$ to equal $\d _0$ for $|x|< R_2$ and $W$ otherwise. For $\im z_0 > 0$ fixed (see below) the operator
$$
\mbf{E}(z,z_0) = \widetilde{\mbf{\h }}_1(\ham{D}-z_0)^{-1}\mbf{\h }_1 + \widetilde{\mbf{\h }}_2(\ham{D}_0-i\ham{W}-z_0)^{-1}\mbf{\h }_2 + \widetilde{\mbf{\h }}_3(\ham{D}_0-i\ham{W}_1-z)^{-1}\mbf{\h }_3
$$
depends analytically on $z$ in $\im z > - \d _0$. Moreover 
$$
(\ham{J}-z) \mbf{E}(z,z_0) = \mbf{1}+\mbf{K}(z,z_0)
$$
where 
\begin{align*}
\mbf{K}(z,z_0) &= [\ham{D}_0,\widetilde{\mbf{\h }}_1](\ham{D} - z_0)^{-1}\mbf{\h}_1 + (z_0-z)\widetilde{\mbf{\h }}_1(\ham{D} - z_0)^{-1}\mbf{\h}_1 \\
&+[\ham{D}_0,\widetilde{\mbf{\h }}_2](\ham{D}_0 - i\ham{W} - z_0)^{-1}\mbf{\h}_2 + (z_0-z)\widetilde{\mbf{\h }}_2(\ham{D}_0 - i\ham{W} - z_0)^{-1}\mbf{\h}_2 \\
&+[\ham{D}_0,\widetilde{\mbf{\h }}_3](\ham{D}_0 - i\ham{W}_1 - z)^{-1}\mbf{\h}_3 \\
&=:\mbf{K}_1(z_0) +  \mbf{K}_2(z,z_0) + \mbf{K}_3(z_0)  + \mbf{K}_4(z,z_0) + \mbf{K}_5(z).
\end{align*}
By construction $\mbf{K}(z,z_0)$ depends analytically on $z$ in $\im z > -\d _0$. Furthermore it follows by the Rellich-Kondrachov embedding theorem that $\mbf{K}(z,z_0)$ is a compact operator on $\Lb ^2(\R ^3,\C ^4)$. Since for $\im z_0>0$ sufficiently large we have $\|\mbf{K}(z,z_0) \|\le C\max \{|\im z|^{-1}, (\im z_0)^{-1} \}$ we see that for $z=z_0$ and $\im z_0$ large enough $\|\mbf{K}(z,z_0) \| \le 1/2$. By the analytic Fredholm theorem, for fixed $z_0$ as above, $(\mbf{1} + \mbf{K}(z,z_0) )^{-1}$ exists as a meromorphic operator in $\im z > -\d _0$. Similarly a left parametrix is constructed by interchanging $\mbf{\h }_j $ and $\widetilde{\mbf{\h }}_j$ for $j=1,2,3$. The left and right inverses will share the same poles and agree elsewhere and thus
\begin{align}
(\ham{J} - z)^{-1} = \mbf{E}(z,z_0)(\mbf{1} + \mbf{K}(z,z_0))^{-1}
\label{J-resolvent-with-K}
\end{align}
so that $(\ham{J} - z)^{-1}$ is meromorphic in $\im z > - \d _0$ with finite rank residues at the poles which are the eigenvalues. 
\end{proof}
\noindent
The following result and its proof is similar to \cite[Proposition 2]{stefanov05}. 
\begin{proposition}
Let Assumption \ref{jkmm2:perturd-assumV} and Assumption \ref{jkmm2:capham-assumonw} hold. If $\Rc $ satisfies Assumption $(\textbf{A}_{\Rc }^+)$ then the number of eigenvalues in $\Rc $ satisfies 
\begin{align}
\coun (\ham{J}(\hbar ), \Rc ) = \Oc (\hbar ^{-4}). 
\label{upper-bound-eigenvalues-J-prop}
\end{align}
\label{jkmm2:aux-prop1}
\end{proposition}
\begin{proof}
In addition to the requirements imposed in the proof of Lemma \ref{J-resolvent-meromorphic-delta-lemma}, assume $z_0$ is such that we can find $r_0, \ve _0>0$ so that $\Rc \subset D(z_0,r_0) \subset D(z_0,r_0+\ve _0) \subset \{\im z > -\d_0 \}$. By \eqref{J-resolvent-with-K} it suffices to estimate the number of points $z$ in $D(z_0,r_0)$ where $\mbf{1}+\mbf{K}(z,z_0)$ is not invertible. Since $\|\mbf{K}_5(z) \|=\Oc (\hbar )$ we may write 
\begin{align*}
\mbf{1} + \mbf{K}(z,z_0) &= \big (\mbf{1} + \widetilde{\mbf{K}}(z)\big )  (\mbf{1} + \mbf{K}_1(z_0) + \mbf{K}_3(z_0) + \mbf{K}_5(z)  ) \\
 \widetilde{\mbf{K}}(z,z_0) &:= (\mbf{K}_2(z,z_0) + \mbf{K}_4(z,z_0)) \big (\mbf{1} + \mbf{K}_1(z_0) + \mbf{K}_3(z_0) + \mbf{K}_5(z) \big )^{-1}
\end{align*}
provided $\hbar $ is small enough. Thus $\mbf{1}+\mbf{K}(z,z_0)$ is not invertible if and only if $\mbf{1} + \widetilde{\mbf{K}}(z,z_0)$ is not invertible. Now,  since $\mbf{1} + \widetilde{\mbf{K}}(z,z_0)$ need not belong to $\Bc _1$ (see below) and since the singular points of $\mbf{1} + \widetilde{\mbf{K}}(z,z_0)$ are included among those of $\mbf{1} - \widetilde{\mbf{K}}^4(z,z_0)$, we are going to estimate the number of zeros of 
$$
f(z) := \det (\mbf{1} - \widetilde{\mbf{K}}^4(z,z_0)).
$$
By \eqref{jkmm2:ope-eq3}, \eqref{singular-values-compact-product}, \eqref{singular-values-compact-bounded-1} and \eqref{singular-values-compact-sum} it suffices to obtain upper bounds of $\m _j (\mbf{K}_2)$ and $\m _j (\mbf{K}_4)$. To this end, let $ \widetilde{R} > R_0'+R_1$ and consider the flat torus $\mathbb{T} := (\R / \widetilde{R}\Z )^3$ obtained by identifying opposite faces of the cube $\{x\in \R ^3 : |x_j|< \widetilde{R},\, j=1,2,3 \}$. We assume $\mathbb{T}$ carries the metric induced by the Euclidean metric on $\R ^3$ and trivial spin structure. Denote by $\ham{D}_{0,\mathbb{T}}$ the corresponding free semiclassical Dirac operator on $\mathbb{T}$. Then, viewing $B(0,\tfrac{R_0'+R_1}{2})$ as a subset of $\mathbb{T}$, $\ham{D}_{\mathbb{T}} := \ham{D}_{0,\mathbb{T}} + \ham{V}(x)$ coincides with $\ham{D}$ near $B(0,\tfrac{R_0'+R_1}{2})$. It is well-known that $\ham{D}_{0,\mathbb{T}}$ satisfies the Weyl law (in fact this follows from the Weyl law for $\D _{\mathbb{T}}$ in view of the Schr\"odinger-Lichnerowicz formula)
\begin{align}
\coun (\ham{D}_{0,\mathbb{T}}, [-\l , \l ]) = \Oc \Big ( \frac{\l ^3 }{\hbar ^3} \Big ),
\label{weyl-law-dirac-torus}
\end{align}
and since $\ham{V}$ is a bounded multiplication operator the Weyl asymptotics remain true also for $\ham{D}_{\mathbb{T}}$.
Denote by $\l _1 \le \l _2 \le \cdots $ the eigenvalues of $\ham{D}_{0,\mathbb{T}}$. Then \eqref{weyl-law-dirac-torus} implies  
$$
\m _j \big ((\ham{D}_{\mathbb{T}} - i)^{-1} \big) = |i-\l _j|^{-1} \le \frac{C}{1+\hbar j^{1/3}},
$$   
and by the resolvent equation the same estimate holds for $\m _j \big ((\ham{D}_{\mathbb{T}} - z_0)^{-1} \big)$. From the identity
$$
(\ham{D} - z_0)^{-1} \mbf{\h }_1 = \widetilde{\mbf{\h }}_1(\ham{D}_{\mathbb{T}} - z_0)^{-1}\mbf{\h }_1 -  (\ham{D} - z_0)^{-1}[\ham{D}, \widetilde{\mbf{\h }}_1] (\ham{D}_{\mathbb{T}} - z_0)^{-1}\mbf{\h }_1
$$
we now obtain 
$$
\m _j (\mbf{K}_2) \le \frac{C}{1+\hbar j^{1/3}}.
$$
By taking a possibly larger torus we see that 
$$
\widetilde{\mbf{\h }}_2(\ham{D}_0 - i\ham{W} - z_0)^{-1}\mbf{\h}_2 = (\ham{D}_{0,\mathbb{T}} - i)^{-1}(\ham{D}_{0,\mathbb{T}} - i)\widetilde{\mbf{\h }}_2(\ham{D}_0 - i\ham{W} - z_0)^{-1}\mbf{\h}_2
$$
where $(\ham{D}_{0,\mathbb{T}} - i)\widetilde{\mbf{\h }}_2(\ham{D}_0 - i\ham{W} - z_0)^{-1}\mbf{\h}_2$ is bounded. Thus, by \eqref{weyl-law-dirac-torus}, also 
$$ 
\m _j (\mbf{K}_4) \le \frac{C}{1+\hbar j^{1/3}}. 
$$
It follows that 
$$
\sum _j \m _j (\widetilde{\mbf{K}}^4) \le \sum _j \frac{C}{(1+\hbar j^{1/3})^4} \le \sum _j \frac{C}{1+\hbar ^4 j^{4/3}}\le C\hbar ^{-4},
$$
and from \eqref{jkmm2:ope-eq3} we obtain $|f(z)|\le e^{C\hbar ^{-4}}$ for $z\in D(z_0,r_0 + \ve _0)$. Thus, since $f(z_0)=1$, an application of Jensen's formula relative to $D(z_0,r_0 + \ve _0)$ and $D(z_0,r_0)$ gives \eqref{upper-bound-eigenvalues-J-prop}.  
\end{proof}
Finally we establish an a priori resolvent estimate for the complex scaled CAP Hamiltonian $\ham{J}_\th $, which takes into account the distance to its eigenvalues $w_{j}$;  this is the analogue of Proposition~\ref{jkmm2:aux-prop2} above. 
 
\begin{proposition}
Let Assumption~\ref{jkmm2:perturd-assumV} and Assumption~\ref{jkmm2:capham-assumonw} hold. Let $\Rc $ be a complex rectangle 
satisfying Assumption $(\textbf{A}_{\Rc }^+)$ and assume $g : (0,\hbar _0] \to \R _+$ is $o(1)$. Then there are constants 
$A=A(\Rc )>0$ and $\hbar _1 \in (0,\hbar _0)$ such that
\begin{equation*}
\| (\ham{J}(\hbar)-z)^{-1} \| \leq A e^{A \hbar^{-4} \log \frac{1}{g(\hbar)} } , \quad z \in \Rc  \setminus \bigcup_{w_{j}(\hbar ) \in \spec(\ham{J}_{\th }(\hbar)) \cap \Rc^{\prime}}  D(w_{j}(\hbar ), g(\hbar) ),
\end{equation*}
where $\Rc \subsetneq \Rc^{\prime}$.
\label{jkmm2:aux-prop3}
\end{proposition}

The following proof is partially sketchy to avoid repeating arguments.

\begin{proof}
In this proof, once again, we suppress the subscript in $\ham{J}_{\infty}(\hbar)$ and its dependence on $\hbar$. Using the notation from Lemma \ref{J-resolvent-meromorphic-delta-lemma} and Proposition \ref{upper-bound-eigenvalues-J-prop} we have
$$
(\ham{J} - z)^{-1} = \mbf{E}(\mbf{1} + \mbf{K}_1 + \mbf{K}_3 + \mbf{K}_5)^{-1} (\mbf{1} - \widetilde{\mbf{K}} + \widetilde{\mbf{K}}^2 - \widetilde{\mbf{K}}^3)(\mbf{1} - \widetilde{\mbf{K}}^4 )^{-1}
$$ 
so it suffices to estimate $(\mbf{1} - \widetilde{\mbf{K}}^4 )^{-1}$ away from the set of eigenvalues of $\ham{J}$. To this end we have (see \cite[Ch. V, Theorem 5.1]{gohkre69})
\begin{equation*}
  \| (\mbf{1} + \widetilde{\mbf{K}}^4(z))^{-1} \| \leq \frac{\det (\mbf{1} + |\widetilde{\mbf{K}}^4(z)|)}{|\det (\mbf{1} + \widetilde{\mbf{K}}^4(z))|}.
\label{jkmm2:aux-prop3-eq3}
\end{equation*}  
For the numerator we have as before $\det (\mbf{1} + |\widetilde{\mbf{K}}^4(z)|) \leq e^{\| \widetilde{\lmbf{K}}(z) \|_{\Bc_{1}}} \leq e^{C\hbar ^{-4 }}$. The denominator can be treated as in \cite[Section 8]{sjo97b}, 
i.e. by first factoring out its zeros and then use the upper bound for the eigenvalue counting function to obtain 
\begin{equation*}
  |\det (\mbf{1} + \widetilde{\mbf{K}}_0(z))| \geq Ce^{C\hbar ^{-4} \log \frac{1}{g}} \quad \text{for }\dist (z, \spec(\ham{J}_{\th}) \cap \Rc )\geq g(\hbar ).
\label{jkmm2:aux-prop3-eq4}
\end{equation*}
Putting these facts together gives the assertion. 
\end{proof}

\begin{remark}
The results above, established for $\ham{J}_{\infty}(\hbar)$ and its resolvent, can easily be carried over to the CAP Hamiltonian 
$\ham{J}_{R}(\hbar)$ and its resolvent. 
\label{jkmm2:aux-remark2}
\end{remark} 
\section{Quasimodes and resonances}
\label{jkmm2:quasi} 
In this section we present the main result that enables us to relate so-called \textit{quasimodes} of $\ham{D}(\hbar)$ with resonances of $\ham{D}(\hbar)$. 
It informs us that if we have a set of linearly independent quasimodes, which can be thought of as square integrable approximate 
resonant states, for energies in a real interval $I$ and if this set remains linearly independent under small perturbations 
(in the semiclassical sense), then there are as many resonances as there are quasimodes and these are located with real parts near $I$ 
and having small imaginary parts. Such a result was first established by Tang and Zworski \cite{tangzwor98} for Schr\"{o}dinger operators.  
We give a version which is valid for the perturbed Dirac operator.
Our proof is adopted from Stefanov \cite{stefanov99} who even managed to treat higher multiplicities and clusters of resonances in the case 
when quasimodes are very close to each other. He showed that such clusters of quasimodes generate (asymptotically) at least the same 
number of resonances. In \cite{stefanov05} he improved the latter result in several ways by modifying the reasoning in \cite[Theorem 1]{stefanov99}. 
The underlying ideas, however, are the same as in Tang and Zworski  \cite{tangzwor98}  (see also \cite[Theorem 11.2]{sjo02}).

Let $\h ,\widetilde{\h } \in \ccs (\R^3)$ with $\mathbf{1}_{B(0,R)} \prec \h \prec \widetilde{\h }$ and let $z_0\in \Res(\ham{D}(\hbar))$. Then, for $z$ in 
a neighborhood of $z_0$ we have, with $N$ finite, 
\begin{equation}
\mbf{\h} (\ham{D}_{\th} - z)^{-1} \widetilde{\mbf{\h} } = \mbf{A}_0 (z, \hbar ) + \sum _{j=1}^N (z - z_0 (\hbar ))^{-j}\mbf{A} _j(\hbar )
\label{jkmm2:quasi-eq1}
\end{equation}
for some operator $\mbf{A}_0(z, \hbar )$, holomorphic in $z$ near $z_0(\hbar )$, and finite rank operators $\mbf{A}_j$, $1 \le j \le N$, independent of $z$.      

\begin{lemma}
Let $\h \in \ccs (\R ^3)$ with $\h = 1$ on $B(0,R)$ for some $R>0$. Then, for any $z_0(\hbar ) \in \Res (\ham{D}(\hbar ))$, we have 
\begin{align*}
\mbf{\h} (\ham{D}_{\th} (\hbar ) - z)^{-1} \mbf{\h} = \mbf{A}_0(z,\hbar )\mbf{\h} + \sum _{j=1}^N (z - z_0(\hbar ))^{-j} \mbf{A}_1(\hbar ) \mbf{Q}_j (\hbar )
\label{jkmm2:quasi-lem1-eq1}
\end{align*}
for some operators $\mbf{Q}_j$, holomorphic at $z_0(\hbar )$. 
\label{jkmm2:quasi-lem1}
\end{lemma}

\begin{proof}
For notational reasons we denote $\h = \h _1$ and $\widetilde{\h } = \h _N$ and introduce the sequence of intermediate cut-off functions 
\begin{equation*}
\h _1 \prec \h _2 \prec \cdots \prec \h _N.
\label{jkmm2:quasi-lem1-eq2}
\end{equation*}
Multiply (\ref{jkmm2:quasi-eq1}) by $\ham{D}_{\th} - z$ from the right to get 
\begin{align*}
\mbf{\h}_1+ \mbf{\h}_1 & (\ham{D}_{\th} - z)^{-1} [\widetilde{\mbf{\h} }, \ham{D}_{\th}] = \mbf{A}_0 (z) (\ham{D}_{\th} - z) + 
\sum _{j=1}^N (z - z_0)^{-j} \mbf{A}_j (\ham{D}_{\th} - z) \\ 
&=\mbf{A}_0(z) (\ham{D}_{\th} - z) - \mbf{A}_1 + \sum _{j=1}^N (z - z_0)^{-j} \big ( \mbf{A}_j (\ham{D}_{\th} - z) - \mbf{A}_{j+1} \big )
\label{jkmm2:quasi-lem1-eq3}
\end{align*}
with the convention that $\mbf{A}_{N + 1} = 0$. Upon multiplying by $\mbf{\h}_l$ from the right and using $[\widetilde{\mbf{\h} }, \ham{D}_{\th}] \mbf{\h}_l = 0$ 
we realize that all singular terms on the right must vanish, that is
\begin{equation*}
\mbf{A}_j (\ham{D}_{\th} - z_0) \mbf{\h}_l = \mbf{A}_{j+1} \mbf{\h}_l, \quad  j,l = 1, \ldots ,N-1.
\label{jkmm2:quasi-lem1-eq4}
\end{equation*}
Using the latter identity repeatedly results in 
\begin{align*}
\mbf{A}_j \h _1 &= \mbf{A}_{j-1} (\ham{D}_{\th} - z_0) \mbf{\h}_1 = \mbf{A}_{j-1} \mbf{\h}_2( \ham{D}_{\th} - z_0) \mbf{\h}_1 \\
 &\phantom{a}\vdots \\
&= \mbf{A}_1( \ham{D}_{\th} - z_0) \mbf{\h}_{j-1} ( \ham{D}_{\th} - z_0) \mbf{\h}_{j-2} \cdots \mbf{\h}_2 ( \ham{D}_{\th} - z_0) \h _1.
\label{jkmm2:quasi-lem1-eq5}
\end{align*} 
By multiplying (\ref{jkmm2:quasi-eq1}) from the right by $\mbf{\h} $ and using the previous relation we obtain the lemma with 
\begin{equation*}
\mbf{Q}_j = (\ham{D}_{\th} - z_0) \mbf{\h}_{j-1} (\ham{D}_{\th} - z_0) \mbf{\h} _{j-2} \cdots \mbf{\h} _2 (\ham{D}_{\th} - z_0) \h _1. \qedhere  
\label{jkmm2:quasi-lem1-eq6}
\end{equation*}  
\end{proof}

We state and prove the main result of this section for positive energies. 

\begin{theorem}
Assume $mc^2 < l_0\le l(\hbar ) \le r(\hbar )\le r_0 <\infty $. Assume that for any $\hbar \in (0,\hbar _0]$ there is $m(\hbar ) \in \Z _+$, $E_j(\hbar ) \in [l(\hbar ), r(\hbar )]$ and normalized $\mbf{u}_j(\hbar ) \in \dom (\ham{D})$ (quasimodes) for $1\le j \le m(\hbar )$, having support in a ball $B(0,R)$ where $R<R_0$ does not depend on $\hbar $. Assume, moreover, that
\begin{align}
&\|( \ham{D} (\hbar ) - E_j(\hbar ))\mbf{u}_j(\hbar ) \| \leq \rho(\hbar ) 
\label{jkmm2:quasi-thm1-eq1}
\intertext{and}
&\text{all } \tilde{\mbf{u}}_j(\hbar ) \in \Hc \text{ such that } \|\tilde{\mbf{u}}_j(\hbar ) - \mbf{u}_j(\hbar ) \| \leq \frac{\hbar ^N}{M}, \quad 1\le j \le m(\hbar ), \nonumber \\
&\text{are linearly independent}, 
\label{jkmm2:quasi-thm1-eq2}
\end{align} 
where $\rho(\hbar ) \leq \hbar ^{4+N}/(C\log \hbar ^{-1})$, $C\gg 1$, $N\ge 0$ and $M>0$. Then there exists $C_0 = C_0(l_0,r_0)>0$ 
such that for any $B>0$ and $K\in \Z _+$ there is an 
$\hbar _1 = \hbar _1(A,B,M,N) \le \hbar _0$ such that for any $\hbar \in (0,\hbar _1]$ there will be at least $m(\hbar )$ resonances of 
$\ham{D}(\hbar)$ in 
\begin{align}
[l(\hbar ) - b(\hbar )\log \frac{1}{\hbar }, r(\hbar ) + b(\hbar )\log \frac{1}{\hbar }] + i[-b(\hbar ),0],
\label{jkmm2:quasi-thm1-eq3}
\end{align}
where 
\begin{equation*}
b(\hbar ) = \max {(C_0BM \rho(\hbar ) \hbar ^{-4-N}, e^{-B/\hbar }, \hbar ^K)}.
\label{jkmm2:quasi-thm1-eq4}
\end{equation*}
\label{jkmm2:quasi-thm1}
\end{theorem}
We remark that Theorem~\ref{jkmm2:quasi-thm1} is stronger than what is needed for the present work where we only work with one quasimode at a time.  
\begin{proof}
Denote by $z_1, \ldots ,z_l$ all \emph{distinct} resonances in 
\begin{align}
\Rc_2 := [l-2w, r+2w] + i\Big [-2A\hbar ^{-3} (\log \frac{1}{S}) S, S \Big ].
\label{jkmm2:quasi-thm1-eq5}
\end{align}
where $S = \max (e^3 M \hbar ^{-N} \rho(\hbar ), e^{-2B/\hbar }, \hbar ^{K+4})$ for some $B>0$ (cf.\ Appendix~\ref{jkmm2:app-3}) and 
\begin{equation*}
w = 12 A \hbar ^{-3} (\log \frac{1}{\hbar }) (\log \frac{1}{S})S.
\label{jkmm2:quasi-thm1-eq6}
\end{equation*}
Clearly $S$ and $w$ satisfies (\ref{jkmm2:app-prop3-eq2}). It is easy to see that $\Rc_{2} \cap \C_{-}$ is contained in the box 
(\ref{jkmm2:quasi-thm1-eq3}) for $\hbar $ small enough so it suffices to show that there are at least $m$ resonances in $\Rc _2$.  
Fix $\h \in \ccs (\R ^3)$ with $\h \succ \mathbf{1}_{B(0, R)}$. Let $\mbf{\Pi} $ be the orthogonal projection onto $\cup _j \mbf{A}_{1}^{(j)}(\Hc )$ 
with $\mbf{A}_{1}^{(j)}$ being the residue at $z_{j}$, cf.~(\ref{jkmm2:quasi-eq1}),  and let $\mbf{\Pi}^{\prime} = \mbf{1} - \mbf{\Pi}$ be the 
complementary projection. In view of Lemma~\ref{jkmm2:quasi-lem1} $\mbf{F}(z) := \mbf{\Pi}^{\prime} \mbf{\h} (\ham{D}_{\th}(\hbar)-z)^{-1} \mbf{\h}$ is 
holomorphic in a neighborhood of $\Rc_2$. We are going to use this fact to show that the estimate in (\ref{jkmm2:aux-prop2-eq1}) holds in the 
whole of the smaller box 
\begin{equation*}
\Rc _1:= [l - w, r + w]+i\Big [-A\hbar ^{-3} (\log \frac{1}{S}) S, S \Big ] \subset \Rc_{2}
\label{jkmm2:quasi-thm1-eq7}
\end{equation*}
The bound $\| \mbf{F}(z)\| \leq C/ S$ (cf.\ Proposition~\ref{jkmm2:app-prop3}) for $\im z = S$ follows from Proposition \ref{jkmm2:cscaled-resolv-upper}.  
From Proposition~\ref{jkmm2:aux-prop2} with $g = S$ it follows that (\ref{jkmm2:aux-prop2-eq1}) is fulfilled for 
$z\in \Rc_2 \cap \{z:\dist (z, \Res (\ham{D})) \ge S \}$. Consider now the set obtained by adjoining to $\Rc _1$ the set of unions of disks $D(z_j, S)$ that 
have a point in common with $\Rc _1$ (see Figure \ref{quasi_to_res_fig1}).       
\begin{figure}[ht!]
\begin{center}
\def\svgwidth{0.8\columnwidth}
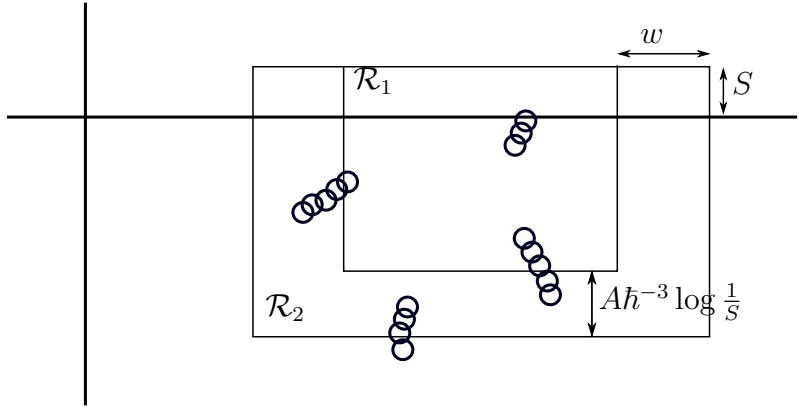
\end{center}
\caption{Connected unions of disks centered at resonances with radius $S$ that intersect with $\Rc _1$ never intersect the complement of $\Rc_2$.}
\label{quasi_to_res_fig1}
\end{figure}

If we can show that the set so obtained is contained in $\Rc_2$, provided $\hbar $ is small enough, where $\mbf{F}$ is holomorphic, then it would follow from the (classical) maximum principle that (\ref{jkmm2:aux-prop2-eq1}) holds in all of $\Rc _1$ since we know it holds on the boundary of the extended set. To this end, notice how it follows from Theorem~\ref{jkmm2:restate-thm1} that the diameter of any connected chain of disks centered at resonances having radii $S$ is $\Oc (\hbar ^{-3} S)$ while the shortest distance from $\Rc _1 \cap \C _-$ (to where the resonances are confined) to the complement of $\Rc_2$ is $A \hbar ^{-3} (\log S^{-1}) S$. Since the latter is greater than the former, provided $\hbar $ is sufficiently small, it follows that any such union of disks that intersect with $\Rc _1$ cannot intersect the complement of $\Rc_2$. Thus 
\begin{equation*}
\| \mbf{F}(z) \| \le A e^{A\hbar ^{-3} \log \tfrac{1}{S}} \quad \text{for all } z\in \Rc _1.
\label{jkmm2:quasi-thm1-eq8}
\end{equation*} 
We are now in a position to apply Proposition~\ref{jkmm2:app-prop3} so that, by letting $z\to E_j$,
\begin{align*}
\| \mbf{\Pi} \mbf{u}_j - \mbf{u}_j \| = \| \mbf{\Pi} '  \mbf{\h} \mbf{u}_j \| = \| \mbf{F}(z) (\ham{D} - z) \mbf{u}_j \| \leq \frac{e^3 \rho(\hbar)}{S}, 
\label{jkmm2:quasi-thm1-eq9}
\end{align*}
where we have also used the quasimode property (\ref{jkmm2:quasi-thm1-eq1}). It follows from our choice of $S$ and the assumption 
(\ref{jkmm2:quasi-thm1-eq2}) that $\{ \mbf{\Pi} \mbf{u}_j \}_{j=1}^m$ is linearly independent. Consequently,
\begin{equation*}
\coun(\ham{D}, \Rc_{2} ) = \sum _{j = 1}^l \rank \mbf{A}_{1}^{(j)} \ge \rank \mbf{\Pi} \ge m, 
\label{jkmm2:quasi-thm1-eq10}
\end{equation*} 
which concludes the proof. 
\end{proof}

With minor modifications Theorem~\ref{jkmm2:quasi-thm1} holds also with $\ham{D}_\th (\hbar)$ replaced by $\ham{J}(\hbar)$ (and resonances by eigenvalues). We re-phrase it for the precise statement.
\begin{corollary}
Assume $mc^2 < l_0\le l(\hbar ) \le r(\hbar )\le r_0 <\infty $. Assume that for any $\hbar \in (0,\hbar _0]$ there is $m(\hbar ) \in \Z _+$, $E_j(\hbar ) \in [l(\hbar ), r(\hbar )]$ and normalized $\mbf{u}_j(\hbar )$ (quasimodes) for $1\le j \le m(\hbar )$, having support in a ball $B(0,R)$ where $R<R_0$ does not depend on $\hbar $. Assume, moreover, that
\begin{align}
&\|( \ham{J} (\hbar ) - E_j(\hbar ))\mbf{u}_j(\hbar ) \| \leq \rho(\hbar ) 
\label{jkmm2:quasi-corol1-eq1a}
\intertext{and}
&\text{all } \tilde{\mbf{u}}_j(\hbar ) \in \Hc \text{ such that } \|\tilde{\mbf{u}}_j(\hbar ) - \mbf{u}_j(\hbar ) \| \leq \frac{\hbar ^N}{M}, \quad 1\le j \le m(\hbar ), \nonumber \\
&\text{are linearly independent}, 
\label{jkmm2:quasi-corol1-eq2a}
\end{align} 
where $\rho(\hbar ) \leq \hbar ^{5+N}/(C\log \hbar ^{-1})$, $C\gg 1$, $N\ge 0$ and $M>0$. Then there exists $C_0 = C_0(l_0,r_0)>0$ 
such that for any $B>0$ and $K\in \Z _+$ there is an 
$\hbar _1 = \hbar _1(A,B,M,N) \le \hbar _0$ such that for any $\hbar \in (0,\hbar _1]$ there will be at least $m(\hbar )$ eigenvalues of $\ham{J}(\hbar)$ in 
\begin{align}
[l(\hbar ) - b(\hbar )\log \frac{1}{\hbar }, r(\hbar ) + b(\hbar )\log \frac{1}{\hbar }] + i[-b(\hbar ),0],
\label{jkmm2:quasi-corol1-eq3_2}
\end{align}
where 
\begin{equation*}
b(\hbar ) = \max {(C_0BM \rho(\hbar ) \hbar ^{-5-N}, e^{-B/\hbar })}.
\end{equation*}
\label{jkmm2:quasi-corol1}
\end{corollary}

\begin{proof}
This proof works just as above if we use  $\coun(\ham{J}(\hbar),\Rc )=\Oc (\hbar ^{-4})$ (see Proposition \ref{jkmm2:aux-prop1}) and define $w$ accordingly. From 
\begin{align}
-\im \langle (\ham{J}-z)\mbf{u}, \mbf{u} \rangle &= \| \sqrt{\re (W)}\mbf{u} \|^2 + \im z \|\mbf{u}\|^2 \ge \im z \|\mbf{u}\|^2
\label{jkmm2:quasi-corol1-eq1}
\end{align}
it follows that 
\begin{align}
\| (\ham{J} - z)^{-1}\|\le \frac{1}{\im z} \quad \text{for } \im z > 0. 
\label{J-resolvent-self-adjoint-estimate}
\end{align}
Therefore, by Proposition \ref{jkmm2:aux-prop3}, we can apply the semiclassical maximum principle to the 
holomorphic function $\mbf{F}(z)=(\mbf{1} -\mbf{\Pi} )(\ham{J} - z )^{-1}\mbf{\h}$, where $\mbf{\Pi}$ is defined similarly as before.
\end{proof}
\begin{remark}
Notice that due to \eqref{J-resolvent-self-adjoint-estimate} we do not need to take $c(\hbar )$ as large as in Theorem \ref{jkmm2:quasi-thm1} which possibly gives us an improved error estimate. 
\end{remark}
\section{Proof of main results}
\label{jkmm2:prf}

\subsection{Approximating a single eigenvalue when $R_0' < R_1$}
\label{jkmm2:prfind1}
We are now in position to prove that a single resonance of $\ham{D}(\hbar )$ generates a single eigenvalue of $\ham{J}(\hbar )$ 
nearby in the sense stated in Theorem~\ref{jkmm2:resultindthm1}, and vice versa. To show these results we use Theorem \ref{jkmm2:quasi-thm1} and Corollary \ref{jkmm2:quasi-corol1}, respectively, with $m=1$. 

\begin{proof}[Proof of Theorem~\ref{jkmm2:resultindthm1}] We suppress the dependence of $\hbar$ for all operators below.
\newline
\noindent
1.  Take $\h \in \ccs (B(0,R_{1}))$ with $\h = 1$ in a neighborhood of $B(0,R_0')$. Let 
$\mbf{u}$ be an eigenfunction of $\ham{D}_\th $ associated with the eigenvalue $z_0$. Since $\mbf{\h} \ham{W} = \mbf{0}$ we have 
\begin{equation*}
(\ham{J} - \re z_0) \mbf{\h} \mbf{u} = [\ham{D}_\th  , \mbf{\h}] \mbf{u} + i \im z_0 \mbf{\h} \mbf{u}.
\label{jkmm2:resultindthm1-eq2}
\end{equation*}
Furthermore, it follows from \eqref{jkmm2:ellipticDtheta} that $\|[\ham{D}_\th , \mbf{\h}] \mbf{u}\| = \Oc (\hbar ^\infty )\|\mbf{u}\|$. 
Consequently, 
\begin{equation}
\|( \ham{J} - \re z_0) \mbf{\h} \mbf{u} \| \leq (- \im z_0 +\Oc (\hbar ^\infty ))\| \mbf{u} \|.  
\label{jkmm2:resultindthm1-eq4}
\end{equation}
Another application of \eqref{jkmm2:ellipticDtheta} gives
\begin{equation*}
\| \mbf{u} \| \leq \| \mbf{\h} \mbf{u} \| + \|(\mbf{1}-\mbf{\h}) \mbf{u} \| \leq \| \mbf{\h}  \mbf{u} \| + \Oc (\hbar ^\infty )\| \mbf{u} \| 
\label{jkmm2:resultindthm1-eq5}
\end{equation*}
and, therefore, we obtain, for $\hbar $ sufficiently small, that $(1 - o(1))\| \mbf{u} \| \leq \| \mbf{\h} \mbf{u} \|$,  and thus 
$\| \mbf{u} \| \leq C \| \mbf{\h}  \mbf{u} \|$. Then (\ref{jkmm2:resultindthm1-eq4}) implies that 
\begin{equation*} 
\|( \ham{J} - \re z_0) \mbf{\h}  \mbf{u} \| \leq (- \im z_0 + \Oc (\hbar ^\infty )) \| \mbf{\h} \mbf{u} \|
\label{jkmm2:resultindthm1-eq6}
\end{equation*}
and, by interpreting $\mbf{\h} \mbf{u} / \| \mbf{\h} \mbf{u} \|$ as a quasimode for $\ham{J}(\hbar)$,  an application of 
Corollary~\ref{jkmm2:quasi-corol1} yields
\begin{equation*}
\rho (\hbar ) = -\im z_0 + \Oc (\hbar ^\infty ) \leq \frac{\hbar ^5}{C\log \frac{1}{\hbar }}.
\label{jkmm2:resultindthm1-eq7}
\end{equation*} 
\noindent
2. Let $\mbf{f} \in \Hb^1(\R^{3},\C^4)$ be an eigenvector of $\ham{J}$ corresponding to $w_0$, i.e. $\ham{J} \mbf{f}=w_0 \mbf{f}$. Let $\h \in \ccs (B(0,R_0))$, $0\le \h \le 1$, be 1 near $B(0,R_2)$. We will show that $\mbf{\h} \mbf{f}  / \| \mbf{\h} \mbf{f} \|$ 
is a quasimode. Therefore we consider 
\begin{align}
(\ham{D} - \re w_0) \mbf{\h} \mbf{f} = [\ham{D}, \mbf{\h}]\mbf{f} +i\mbf{\h}\ham{W}\mbf{f}+ i \im w_0 \mbf{\h} \mbf{f}.
\label{jkmm2:resultindthm3-eq3}
\end{align}
From \eqref{jkmm2:quasi-corol1-eq1} we have $\|\sqrt{\re W}\mbf{f}\| = \sqrt{-\im w_0}\|\mbf{f}\|$, which because of Assumption \ref{jkmm2:capham-assumonw} (iv) and the fact that $[\ham{D}, \mbf{\h}]\mbf{f}$ is supported in $|x|>R_2$, makes the norms of the first two terms on the right hand side bounded by $C\sqrt{-\im w_0}\|\mbf{f}\|$. For the same reason $\mbf{\h} \mbf{f}$ is uniformly bounded away from zero and $\| \mbf{f}\| \le C \| \mbf{\h} \mbf{f} \|$. Thus    
\begin{equation*}
\| (\ham{D} - \re w_0) \mbf{\h} \mbf{f} \| \leq C\sqrt{-\im w_0} \leq \frac{\hbar ^4 }{C\log \frac{1}{\hbar }}.
\end{equation*}
An application of Theorem~\ref{jkmm2:quasi-thm1} finishes the proof.  
\end{proof}

\subsection{Approximating a single eigenvalue when $R_1\leq R_0'$}
\label{jkmm2:prfind3}

\subsubsection{Semiclassical projections}
Denote by $\mbf{\l}_{j} : \Tsf^{\ast} \R^{3} \to \mathrm{M}_{4}(\C )$, $j= 1,\ldots ,4$, the projection matrices onto the eigenspaces corresponding to 
the eigenvalues $\l_{j}$ of the principal symbol $\mbf{d}_0$ of $\ham{D}$. Since the symbols $\mbf{\l}_{j}$ depend on $x$, their 
quantizations $\mbf{\L}_{j}:=\opw{\mbf{\l}_{j}}$ are not projection operators. Rather they satisfy \cite{emmrich96} 
\begin{equation}
\mbf{\L}_{j}^{2} -\mbf{\L}_{j}= \Oc (\hbar), \quad j=1,\ldots ,4.
\label{jkmm2:prfind3-eq3-0}
\end{equation}
In addition, $\mbf{\L}_{j}$ do not commute with $\ham{D}(\hbar)$. One can improve the error on the right-hand side
of (\ref{jkmm2:prfind3-eq3-0}) by adding a suitable term of order $\hbar$ to the symbol $\ham{D}_{0}$ and, subsequently, quantization
results in an operator which is a projector up to an error of order $\hbar^{2}$. Iteration of this process leads to an error of arbitrary order 
$\hbar^{N}$ \cite{emmrich96}. 

Under Assumption~\ref{jkmm2:assum-hyp} it is shown in \cite[Proposition 2.1]{bg04} that one can construct 
\begin{equation*}
\mbf{\l}_{j} \sim \sum _{n\geq 0} \hbar^{n}  \mbf{\l}_{j, n}
\label{jkmm2:prfind3-eq3}
\end{equation*}
such that 
\begin{equation*}
\mbf{\l}_{j} \# \mbf{\l}_{j} \sim \mbf{\l}_{j} \sim \mbf{\l}_{j}^{\ast} \quad \text{and} \quad [\mbf{\l}_{j} , \mbf{d} ]_\# \sim 0, 
\label{jkmm2:prfind3-eq4}
\end{equation*}
and, moreover, in agreement with the discussion above, the corresponding quantizations $\mbf{\L}_{j}$ satisfy the relations ($j=1,\ldots ,4$):
\begin{align*}
  & \mbf{\L}_{1} + \mbf{\L}_{2} \equiv \mbf{1}, \\
  & \mbf{\L}_{j}^{2} \equiv \mbf{\L}_{j}  \equiv \mbf{\L}_{j}^{\ast} ,\\
  & \| [ \ham{D}, \mbf{\L}_{j}] \| \equiv 0,
\label{jkmm2:prfind3-eq4}
\end{align*}
where $\equiv $ means modulo terms of norm $\Oc (\hbar ^\infty )$. The operators $\mbf{\L}_{j}$, $j=1,\ldots ,4$, are called almost 
orthogonal projections.

\subsubsection{Matrix valued Egorov theorem}
We now indicate how to solve Heisenberg's equation of motion semiclassically in the sense that given $\mbf{A}=\opw {\mbf{a}}$ with 
$\mbf{a}\in \Ssf(1)$ we can, for all $t$, find $\mbf{a}(t)\in \Ssf(1)$ such that
\begin{equation*}
  \pdf{}{t}\opw{\mbf{a}(t)} = [\ham{D},\opw{\mbf{a}(t)}] + \mbf{R}(t), \quad \mbf{A}(0) = \mbf{A},
\label{jkmm2:prfind3-eq5}
\end{equation*}
with $\| \mbf{R}(t) \| = \Oc (\hbar ^N)$ for any $N\in \N $. This means we can approximate the time evolution 
$\mbf{A}(t):=\exp(i \ham{D} t/ \hbar ) \mbf{A} \exp(-i \ham{D} t/ \hbar )$ of $\mbf{A}$ to any order.

We extract the following lemma from \cite[Theorem 3.2 and the discussions preceeding and proceeding it]{bg04}.

\begin{lemma}[Matrix Egorov theorem]
Let Assumption~\ref{jkmm2:perturd-assumV} and Assumption~\ref{jkmm2:assum-hyp} be satisfied. Suppose that 
$\mbf{a} \in \Ssf(1)$ is block-diagonal with respect to the $\mbf{\l}_j$ in the sense that
\begin{equation*}
  \mbf{a} \sim \sum _{j=1}^4 \mbf{\l}_j \# \mbf{a}  \# \mbf{\l}_j.
\label{jkmm2:prfind3-lem1-eq1}
\end{equation*}
Then, for any $T>0$, we can find $\mbf{a}(t)\in \Ssf(1)$ for all $0\leq t \le T$ such that 
\begin{equation*}
  \| \mbf{A}(t) - \opw{\mbf{a}(t)} \| = \Oc (\hbar ^\infty ) \quad \text{for all }t\in [0, T].
\label{jkmm2:prfind3-lem1-eq2}
\end{equation*}
Moreover, the principal symbol is given by 
\begin{equation*}
\mbf{a}_0(x,\x ,t) = \sum _{j= 1}^4 \mbf{\tF}_{j j}^\ast (x, \x ,t) \mbf{\l}_{j ,0}(\F_{j}^{t}(x,\x) ) \mbf{a}_0(\F_{j}^t (x,\x )) \mbf{\l}_{j,0}(\F_{j}^{t}(x,\x)) \mbf{\tF}_{j j} (x, \x ,t)
\label{jkmm2:prfind3-lem1-eq3}
\end{equation*} 
where the $4 \times 4$ unitary transport matrices $\mbf{\tF}_{j j }$ are given by 
\begin{equation*}
\frac{d}{dt} \mbf{\tF}_{j j  }(x,\x ,t) +i \widetilde{ \TF}_{j j  ,1}(\F_{j}^{t}(x,\x) ) \mbf{\tF}_{j j  }(x, \x ,t) = 0, \quad \mbf{\tF}(x,\x ,0) = \mbf{I}_{4}.
\label{jkmm2:prfind3-lem1-eq4}
\end{equation*}
Here 
\begin{equation*}
\widetilde{\TF}_{ j j  ,1} = -i \frac{\l _{j}}{2} \mbf{\l}_{j,0} \{ \mbf{\l}_{j ,0}, \mbf{\l}_{j ,0} \} \mbf{\l}_{j ,0}  - 
i  [ \mbf{\l}_{j ,0}, \{\l _{j} , \mbf{\l}_{j ,0} \}] + \mbf{\l}_{j,0} \TF_{j,1} \mbf{\l}_{j ,0}
\label{jkmm2:prfind3-lem1-eq5}
\end{equation*}
where $\TF_{j,1}$ is the subprincipal symbol of $\opw {\mbf{\l}_{j} \# \mbf{d} \#  \mbf{\l}_{j}}$. 
\label{jkmm2:prfind3-lem1}
\end{lemma} 

\begin{remark}
Notice that Lemma~\ref{jkmm2:prfind3-lem1} requires that both \cite[Property (3.9)]{bg04} and the assumptions of 
\cite[Lemma 3.3]{bg04} are satisfied; but these conditions are clearly fulfilled in our case.
\end{remark}

By imposing additional assumptions we can say even more about $\opw{\mbf{a}(T)}$.

\begin{lemma}
Let Assumption~\ref{jkmm2:perturd-assumV}, Definition~\ref{jkmm2:hflow-def} , and Assumption~\ref{jkmm2:assum-hyp} hold. Then, 
provided $T$ and $\hbar ^{-1}$ are sufficiently large we can construct $\mbf{a}(x, \x ,t)$ as in Lemma~\ref{jkmm2:prfind3-lem1} so that 
$\mbf{a}(x, \x , T) = \mbf{I}_{4} + \Oc (\hbar )$ in a neighborhood of $(x_0, \x _0)$. 
\label{jkmm2:prfind3-lem2}
\end{lemma}

\begin{proof}
By defining 
\begin{multline*}
\mbf{a}_{0} (x,\x ) = \sum _{j=1}^4 \mbf{\h}_{j} (x,\x ) \mbf{\l}_{j,0}(x,\x ) \mbf{\tF}_{jj}(\F_{j}^{-T }(x_0, \x _0), T) \mbf{\l}_{j,0}(\F_{j}^{-T }(x_0, \x _0))\\ 
\times \mbf{\tF}_{jj}^{\ast} (\F_{j} ^{-T} (x_0, \x _0), T) \mbf{\l}_{j,0} (x,\x )
\label{jkmm2:prfind3-lem2-eq3}
\end{multline*}  
it follows from Lemma~\ref{jkmm2:prfind3-lem1}, the 
identity (see \cite[Equation (4.1)]{bg04}),
\begin{equation*}
  \mbf{\l}_{j,0}(\F_{j}^{t} (x,\x )) \mbf{\tF}_{jj} (x, \x ,t) \mbf{\l}_{j,0} (x,\x ) = \mbf{\tF}_{jj}(x,\x ,t) \mbf{\l}_{j,0} (x ,\x )
\label{jkmm2:prfind3-lem2-eq4}
\end{equation*}
and its adjoint equation and the fact that $\mbf{\l}_{j,0} (x,\x ) \mbf{\l}_{k,0} (x,\x ) = 0$ whenever $\mbf{\l}_{j,0}$ and $\mbf{\l}_{k,0}$ are projections corresponding to distinct eigenvalues, that
\begin{equation*}
\mbf{a}_{0} (x, \x ,T) = \sum_{j=1}^4 \mbf{\h}_{j}(\F_{j}^{T} (x,\x) ) \mbf{\l}_{j,0}(x,\x )
\label{jkmm2:prfind3-lem2-eq5}
\end{equation*}
and, in particular, $\mbf{a}_{0}(x_0, \x _0,T) = \mbf{I}_{4}$. 
\end{proof}

\subsubsection{Propagation of singularities}

\begin{lemma}[Propagation of singularities]
Let $R_{0}^{\prime} < R_{1}^{\prime}$ and suppose that for some $z_0(\hbar ) \in [l_0,r_0]$ and $\mbf{v}(\hbar ) \in \dom( \ham{D})$ with $\supp \mbf{v} (\hbar ) \subset B(0,R_1^{\prime})$ and $\|\mbf{v}(\hbar ) \| \le C$ for some $C>0$ we have 
\begin{equation*}
( \ham{D} (\hbar ) - z_0(\hbar )) \mbf{v}(\hbar ) = \mbf{g}(\hbar )
\label{jkmm2:prfind3-lem3-eq1}
\end{equation*} 
with $\| \mbf{g}(\hbar ) \| = \Oc ( \ve(\hbar))$, $\ve(\hbar)=\Oc(\hbar^{N})$ for some $N>0$. 
If $(x_0, \x _0)\in \Tsf^{\ast} \R^{3}$ is such that the norms of the $x$-projections of 
$\F_{j}^{T}(x_0, \x _0)$, $j=1,\ldots ,4$, exceed $R_1'$ for some $0<T<\infty $ then $\mbf{v}(\hbar )$ is microlocally $\Oc ( \hbar^{-1} \ve(\hbar ) + \hbar ^\infty)$ at $(x_0, \x _0)$.
\label{jkmm2:prfind3-lem3}
\end{lemma}

\begin{proof}
Let $\mbf{a}(t)\in S (1)$, $0\leq t \leq T$, be as in Lemma~\ref{jkmm2:prfind3-lem1} with $\mbf{a}(0)$ invertible and supported near the points 
$\F_{j}^{T}(x_0 ,\x _0)$ for $j= 1,\ldots ,4$. Denote by $\mbf{A}(t) = \opw{\mbf{a}(t)}$ its quantization. Put 
\begin{equation*}
l (t) = \| \mbf{A}(t) \mbf{v} \|
\label{jkmm2:prfind3-lem3-eq2}
\end{equation*}
so that $l(0) = \Oc (\hbar ^\infty )$ since $\mbf{a}(0)$ has support where $\mbf{v}=0$. Then, using the fact that $\opw {\mbf{a}(t)}$ approximately solves Heisenberg's equation of motion,  
\begin{align*}
    l (t) \frac{d}{dt} l (t) &= \frac{d}{dt} \frac{ l^{2}(t)}{2} = \re \Big \langle  
    \frac{d}{dt}\mbf{A}(t) \mbf{v}, \mbf{A}(t) \mbf{v} \Big \rangle \\
    &= -\hbar ^{-1} \im \Big \langle \big ([\ham{D} - z_0, \mbf{A}(t)] + \mbf{R}(t) \big )\mbf{v}, \mbf{A}(t) \mbf{v} \Big \rangle \\
    &= \hbar ^{-1}\im \Big \langle \mbf{A}(t)( \ham{D} - z_0) \mbf{v}, \mbf{A}(t) \mbf{v} \Big \rangle - 
    \hbar^{-1} \im \Big \langle \mbf{R}(t) \mbf{v}, \mbf{A}(t) \mbf{v}  \Big \rangle \\
    &= \hbar ^{-1}\im \Big \langle \mbf{A}(t)\mbf{g}, \mbf{A}(t)\mbf{v} \Big \rangle - 
    \hbar ^{-1}\im \Big \langle \mbf{R}(t) \mbf{v}, \mbf{A}(t) \mbf{v} \Big \rangle .
\label{jkmm2:prfind3-lem3-eq3}
\end{align*}
Since $\mbf{A}(t)$ is bounded for all $t\in [0,T]$ and $\mbf{R}(t)$ can be made to satisfy 
$\| \mbf{R}(t)æ\mbf{v} \| = \Oc (\hbar ^\infty ) $ we obtain
\begin{equation*}
    \frac{d}{dt} l (t) = \Oc  (\hbar ^{-1 }\ve (\hbar ) + \hbar ^\infty ) 
\label{jkmm2:prfind3-lem3-eq4}
\end{equation*} 
and since $l (0) = \Oc (\hbar ^\infty )$ we see that $l (t) = \Oc (\hbar ^{-1 }\ve (\hbar ) + \hbar ^\infty )$ for all $t \in [0,T]$. In particular, by 
Lemma~\ref{jkmm2:prfind3-lem2}, $\mbf{a}_0(x,\x ,T)$ equals the identity near $(x_0, \x _0)$ so that $\mbf{a}(x,\x ,T)$ is invertible near 
$(x_0, \x _0)$ (see Section~\ref{jkmm2:prelim}) provided $\hbar $ is small enough. 
\end{proof}

\subsubsection{Proof of Theorem~\ref{jkmm2:resultindthm2}}
We are now in position to prove Theorem~\ref{jkmm2:resultindthm2}.

\begin{proof}[Proof of Theorem~\ref{jkmm2:resultindthm2}]
Let $R_{0}^{\prime} < R_1^{\prime} < R_0$ and pick $\h \in \ccs {B(0,R_1')}$ with $\h = 1$ near $B(0,R_0')$. Then, with $\mbf{v} = \mbf{\h } \mbf{u}$, we obtain as in (\ref{jkmm2:resultindthm3-eq3}), 
\begin{equation}
(\ham{J}(\hbar) - \re z_0) \mbf{v} = [\ham{D}(\hbar), \mbf{\h} ] \mbf{u} + i (\im z_0) \mbf{v} - i \ham{W} \mbf{v},
\label{jkmm2:resultindthm2-eq2}
\end{equation} 
where, as in \eqref{jkmm2:resultindthm1-eq4}, $\| (\ham{D}(\hbar)- \re z_0) \mbf{v} \| = \Oc (\ve (\hbar ))$ with $\ve (\hbar ) =  - \im z_{0}(\hbar) + \Oc (\hbar ^\infty )$.

It remains to estimate $\|\ham{W} \mbf{v}\|$. To this end, notice that under the nontrapping assumption (see Definition \ref{jkmm2:hflow-def}) Lemma \ref{jkmm2:prfind3-lem3} can be applied to any point $(x_0, \x _0) \in E_{[l_0, r_0]}$ with $x_0 \in \supp \ham{W}$ where $E_{[l_0, r_0]} = \cup _{j=1}^4 \l  _j ^{-1} ([l_0, r_0])$. Near any point $(x, \x )$ in the complement of $E_{[l_0, r_0]}$ the operator $\ham{D} - z_0$ is elliptic and therefore $\mbf{a}(x,\x ) = \langle \x \rangle ^{-1} \# (\mbf{d}(x, \x) - z_0)$ will satisfy $\|\opw {\mbf{a}} \mbf{v}\| = \Oc (\ve (\hbar )) $ with $\mbf{a}\in \Ssf (1)$ and $\opw {\mbf{a}}$ elliptic at $(x, \x )$. Take $\h \in \cs (\Tsf ^\ast \R ^3)$ which equals 1 near $E_{[l_0, r_0]}$ and has support in $\{|\x | \le C \}$ for some $C>0$. Consider any $(x_0, \x _0)$ with $R_1<|x_0|<R_1'$. For any $(x _0, \x _\a ) \in \{ x_0 \}\times \{|\x | \le C \}$ and any $\mbf{b}_\a \in \Ssf (1)$ supported in a sufficiently small open neighborhood $U_\a $ of $(x _0, \x _\a )$ it holds, by Lemma \ref{jkmm2:prfind3-lem3} and Lemma \ref{small_support}, that $\|\opw {\mbf{b}_\a } \mbf{v}\| = \Oc (\hbar ^{-1}\ve (\hbar ) + \hbar ^\infty )$. Now consider any $\mbf{b} \in \ccs (\cup U_\a )$ and extract a finite subcover $\{U_j\}_{j=1}^N$ of $\supp \mbf{b}$ and let $\{\mbf{\h} _j \}_{j=1}^N$ be a smooth partition of unity subordinate to it. It follows that
$$
\| \opw {\mbf{b}} \mbf{v} \| \le \sum _{j=1}^N \| \opw{\mbf{\h }_j \mbf{b}}\mbf{v} \| = \Oc (\hbar ^{-1}\ve (\hbar ) + \hbar ^\infty ). 
$$  
Consider 
$$
(\mbf{I}_4 - \mbf{\h }(x, \x)) \# (\mbf{d}(x, \x) - z_0)^{-1} \# (\mbf{d}(x, \x) - z_0) + \mbf{\h }(x, \x) = \mbf {I}_4 + \Oc (\hbar ), 
$$
which is well-defined since $\mbf{\h }(x, \x) = \mbf{I}_4 $ near $E_{[l_0, r_0]}$ where $(\mbf{d}(x, \x) - z_0)^{-1}$ does not exist and is everywhere invertible provided $\hbar $ is small enough [ONE CAN ALSO TAKE $(\mbf{d}^\dagger -z_0)\#(\mbf{d} - z_0) + \mbf{\h}$]. It follows (see Lemma \ref{jkmm2:pseudo-leminv}) that there is $\mbf{q} \in \Ssf (1)$ such that 
\begin{align*}
\mbf{v} = \opw {\mbf{q}} \opw {(\mbf{I}_4 - \mbf{\h })\# (\mbf{d} - z_0)^{-1}\# (\mbf{d} - z_0) + \mbf{\h }}\mbf{v}- \opw{\mbf{r}}\mbf{v},
\end{align*}
with $\|\opw{\mbf{r}}\| = \Oc (\hbar ^\infty )$. 

Pick $\h _0 = \h _0 (x) \in \ccs ( \pi _x (\cap _{j=1}^N U_j))$ which equals 1 in a neighborhood $V$ of $x_0$.
Here $\pi _x (x_0,\x _0) = x_0$ denotes the projection of $\Tsf ^\ast \R ^3$ onto its base manifold. Then
$$
\mbf{\h }_0 \mbf{v} = \opw {\mbf{\h } _0 \# \mbf{q} \# \mbf{\h } } \mbf{v} + \Oc (\ve (\hbar ) + \hbar ^\infty ), 
$$ 
where $\mbf{\h } _0 \# \mbf{q} \# \mbf{\h }$ is asymptotically equivalent to a symbol in $\Ssf (1)$ supported in $\cup U_\a $. We conclude that 
$$
\| \mbf{v} \| _{\Lb ^2 (V)} = \Oc (\hbar ^{-1}\ve (\hbar ) + \hbar ^\infty ).
$$
By compactness of $\{ R_1<|x| < R_1' \}$ we conclude that $\| \mbf{v} \|_{\Lb ^2(\supp \ham{W})} = \Oc (\hbar ^{-1}\ve (\hbar ) + \hbar ^\infty )$. 
It now follows from~\eqref{jkmm2:resultindthm2-eq2} that
\begin{equation*}
\| ( \ham{J}(\hbar) - \re z_0) \mbf{\h}  \mbf{u} \| \leq ( \hbar^{-1} \tilde{\ve }(\hbar ) + \hbar ^\infty )\|  \mbf{\h}  \mbf{u} \|,  
\label{jkmm2:resultindthm2-eq4}
\end{equation*}
so the theorem follows by an application of Corollary~\ref{jkmm2:quasi-corol1}.
\end{proof}
\appendix 
\label{jkmm2:app}

\section{Semiclassical maximum principle}
\label{jkmm2:app-3}
The following result is sometimes referred to as the ``semiclassical maximum principle" \cite{tang_zworski}. The version we
state can be found in Stefanov \cite{stefanov05}.

\begin{proposition}
Let $l(\hbar ) \le r(\hbar )$ and assume that $\mbf{F}(z,\hbar )$ is a holomorphic operator-valued function in a neighborhood of 
\begin{equation}
\Rc (\hbar ) = [l(\hbar ) - w(\hbar ), r(\hbar ) + w(\hbar )] + i\Big [-A\hbar ^{-3}S(\hbar )\log \frac{1}{S(\hbar )}, S(\hbar ) \Big ],
\label{jkmm2:app-prop3-eq1}
\end{equation}
where $e^{-B/\hbar } < S(\hbar ) <1$ for some $B>0$ and 
\begin{equation}
3A\hbar ^{-3}S(\hbar )(\log \frac{1}{\hbar })(\log \frac{1}{S(\hbar )}) \le w(\hbar ).
\label{jkmm2:app-prop3-eq2}
\end{equation} 
If moreover $\mbf{F}(z,\hbar )$ satisfies 
\begin{align}
\|\mbf{F}(z,\hbar ) \| &\le e^{A\hbar ^{-3} \log (1/S(\hbar ))}
\intertext{and}
\|\mbf{F}(z,\hbar )\| &\le \frac{C}{\im z} \qquad \text{on }\Rc (\hbar ) \cap \{ \im z = S(\hbar )\},
\label{jkmm2:app-prop3-eq3}
\end{align}
there there exists $\hbar _1 = \hbar _1 (S)>0$ such that 
\begin{equation}
\| \mbf{F}(z,\hbar ) \| \leq \frac{e^3}{S(\hbar )}, \quad \text{for all }z\in [l(\hbar ), r(\hbar )] + i[-S(\hbar ), S(\hbar )],
\label{jkmm2:app-prop3-eq4}
\end{equation}
for all $0<\hbar \leq \hbar _1$.
\label{jkmm2:app-prop3}
\end{proposition}

\end{document}